\documentclass[11pt]{article}

\usepackage[T1]{fontenc}
\usepackage[utf8]{inputenc}
\usepackage[margin=1in]{geometry}
\usepackage{setspace}
\usepackage{newtxtext}
\usepackage{newtxmath}

\usepackage{amsmath}
\usepackage{amssymb}
\usepackage{amsfonts}
\usepackage{amsthm}
\usepackage{mathtools}
\usepackage{bm}
\usepackage{bbm}

\usepackage{graphicx}
\usepackage{subcaption}
\usepackage{booktabs}
\usepackage{tabularx}
\usepackage{multirow}
\usepackage{array}

\usepackage{enumitem}
\usepackage{algorithm}
\usepackage{algorithmic}

\usepackage{tikz}
\usetikzlibrary{arrows.meta, positioning, shadows, fit, backgrounds, calc}

\usepackage{authblk}

\usepackage[round,compress]{natbib}
\usepackage[dvipsnames]{xcolor}
\usepackage[colorlinks=true,allcolors=green!60!black]{hyperref}

\newcommand{\defeq}{\mathrel{\mathop:}=}

\newcommand{\E}{\mathbb{E}}

\newcommand{\Prob}{\mathbb{P}}

\DeclareMathOperator{\Bel}{Bel}
\DeclareMathOperator{\Pl}{Pl}
\newcommand{\R}{\mathbb{R}}

\newtheorem{theorem}{Theorem}[section]
\newtheorem{lemma}[theorem]{Lemma}
\newtheorem{proposition}[theorem]{Proposition}
\newtheorem{corollary}[theorem]{Corollary}

\theoremstyle{definition}
\newtheorem{definition}[theorem]{Definition}

\newtheorem{example}[theorem]{Example}

\theoremstyle{remark}
\newtheorem{remark}[theorem]{Remark}
 
\title{\textbf{Conformal Prediction $=$ Bayes?}}
\author[1]{Jyotishka Datta}
\author[2]{Nicholas G. Polson}
\author[3]{Vadim Sokolov}
\author[4]{Daniel Zantedeschi}

\affil[1]{Department of Statistics, Virginia Tech, Blacksburg, VA 24061\\ \texttt{jyotishka@vt.edu}}
\affil[2]{Chicago Booth School of Business, University of Chicago, Chicago, IL 60637\\ \texttt{ngp@chicagobooth.edu}}
\affil[3]{Volgenau School of Engineering, George Mason University, Fairfax, VA 22030\\ \texttt{vsokolov@gmu.edu}}
\affil[4]{School of Information Systems, Muma College of Business, University of South Florida, Tampa, FL 33620\\ \texttt{danielz@usf.edu}}

\date{\today}

\begin{document}
\maketitle

\begin{abstract}
Conformal prediction (CP) is widely presented as distribution-free predictive inference with
finite-sample \emph{marginal} coverage under exchangeability. We argue that CP is best
understood as a rank-calibrated descendant of the Fisher--Dempster--Hill
fiducial/direct-probability tradition rather than as Bayesian conditioning in disguise.

We establish four separations from coherent $\sigma$-additive predictive semantics.
First, canonical conformal constructions violate conditional extensionality: prediction sets
can depend on the marginal design $P(X)$ even when $P(Y\mid X)$ is fixed. Second, any
finitely additive sequential extension preserving rank calibration is nonconglomerable,
implying countable Dutch-book vulnerabilities. Third, rank-calibrated updates cannot be
realized as regular conditionals of any $\sigma$-additive exchangeable law on
$\mathcal{Y}^\infty$. Fourth, formalizing both paradigms as families of one-step predictive
kernels, conformal and Bayesian kernels coincide only on a Baire-meagre subset of the space
of predictive laws.

We further show that rank- and proxy-based reductions are generically Blackwell-deficient
relative to full-data experiments, yielding positive Le~Cam deficiency for suitable losses.
Extending the analysis to \emph{prediction-powered inference} (PPI) yields an analogous
message: bias-corrected, proxy-rectified estimators can be valid as confidence devices while
failing to define transportable belief states across stages, shifts, or adaptive selection.
Together, the results sharpen a general limitation of \emph{wrappers}: finite-sample
calibration guarantees do not by themselves supply composable semantics for sequential
updating or downstream decision-making.
\end{abstract}

\noindent\textbf{Keywords:} conformal prediction; predictive inference; exchangeability;
extensionality; conglomerability; Bayesian prediction; predictive kernels; Blackwell
sufficiency; Le Cam deficiency; Baire category; prediction-powered inference; wrappers.
\newpage

\newpage

\newpage
\section{Introduction}
\label{sec:introduction}

Predictive inference in contemporary statistics and machine learning is driven by
two complementary desiderata. One is \emph{coherent probabilistic modeling}:
specify a data-generating mechanism, derive predictive distributions by
conditioning, and use the resulting predictive kernels as state variables for
decision-making. The other is \emph{distribution-free validity}: obtain
finite-sample guarantees that hold under minimal assumptions and remain reliable
under model misspecification. The success of flexible black-box predictors has
amplified demand for the second, making uncertainty quantification a first-order
design objective in learning pipelines.

Conformal prediction (CP) \citep{vovk2005algorithmic,lei2018distribution,angelopoulos2021gentle}
is the flagship method in this distribution-free agenda. Given exchangeable data
$(X_i,Y_i)_{i=1}^{n+1}$, CP constructs a set-valued predictor $C_n(X_{n+1})$
satisfying
\[
\Pr\{Y_{n+1}\in C_n(X_{n+1})\}\;\ge\;1-\alpha,
\]
uniformly over all exchangeable laws and sample sizes. Prediction-Powered
Inference (PPI) \citep{angelopoulos2023prediction,zrnic2024cross,angelopoulos2023ppipp}
addresses a related problem: given a pre-trained predictor, abundant unlabeled
data with proxy predictions, and a small labeled sample for bias correction,
construct valid confidence procedures for population targets without fully
specifying a likelihood. Both methods turn black-box prediction into frequentist
uncertainty quantification---and both, as we show, inherit structural tensions
when their outputs are treated as composable probabilistic states.

A natural question is whether these calibration devices can serve as surrogates
for Bayesian predictive inference. This question is not new. Rather than
revisiting it in purely interpretive terms, we formalize the distinction and
prove separation results that pinpoint where calibration-style validity
diverges from $\sigma$-additive predictive semantics. The underlying tension
between calibration-first constructions and coherent belief states runs through
a century of foundational statistics: Jeffreys' $A_2$ rule \citep{jeffreys1939theory},
contested by Fisher as lacking full probabilistic meaning; Fisher's own fiducial
pivots \citep{fisher1930inverse}, criticized for incoherence under composition;
Dempster's direct probabilities \citep{dempster1963direct}, assigning belief
without a $\sigma$-additive joint; Goldstein's temporal coherence
\citep{goldstein1985temporal}, clarifying what sequential consistency demands;
and the Lane--Sudderth critique \citep{lane1984coherent}, showing that
rank-calibrated rules are nonconglomerable and cannot arise as regular
conditionals of any $\sigma$-additive exchangeable law. CP and PPI inherit this
lineage.

We show that the identification is generically incorrect. The key distinction is
the \emph{type of object} produced. Bayesian prediction yields a sequential
family of predictive kernels---regular conditional distributions forming a
$\sigma$-additive state that supports conditioning, composition, and
decision-theoretic comparison. Conformal prediction yields a \emph{set-valued
rule} with a marginal coverage guarantee; it becomes a distribution only after
an additional kernelization choice. The difference matters when uncertainty must
be transported or updated, and the historical record shows it has been
rediscovered at every stage of the calibration-versus-coherence debate.

\subsubsection*{Implication for practice: wrappers as guardrails, and why this can fail under transport}
For practitioners, the distinction is operational.
If a pipeline only needs a \emph{final-stage validity certificate} (e.g., ``do not
ship a prediction without marginal $1-\alpha$ coverage under the current
exchangeable data stream''), then CP- or PPI-style wrappers are appropriate
guardrails.
If a pipeline needs an \emph{uncertainty state} that will be reused, transported
across covariate shift, conditioned on new side information, optimized inside a
controller, or composed over time, then a set-valued certificate is not a
substitute for a predictive kernel.
In such workflows, place wrappers as terminal constraints on top of a
$\sigma$-additive predictive model (Bayesian or frequentist), and carry forward
the kernel as the object that is updated and compared decision-theoretically.

The reason is that ``distribution-free'' validity is often achieved through
\emph{design-dependent calibration}. In split conformal, calibration pools a
score distribution over the realized covariate marginal; under heteroskedastic
(or more generally covariate-dependent) score laws, the resulting prediction set
depends on $P(X)$ even when $P(Y\mid X)$ is fixed, violating conditional
extensionality (Theorem~\ref{thm:extensionality}). This dependence is exactly
what breaks naive transport, and it is the entry point to the broader
separations in the paper (sequential coherence, $\sigma$-additive realizability,
topological genericity, and experiment-theoretic information loss).

\begin{example}[Transport mismatch in a common two-stage deployment pipeline]
\label{ex:transport_pipeline}
\textit{Stage 1 (source calibration).}
A team trains a regressor and calibrates a split conformal interval under a
source environment $P_S$, producing a rule $x\mapsto C_S(x)$ intended to satisfy
$P_S\{Y\in C_S(X)\}\ge 1-\alpha$.

\textit{Stage 2 (target deployment under shifted design).}
The same model is deployed under a target environment $P_T$ with the same
conditional law $P_T(Y\mid X)=P_S(Y\mid X)$ but a different covariate marginal
$P_T(X)\neq P_S(X)$ (e.g., a new market, a different sensor regime, or a policy
change that shifts which $x$ values are encountered).

\textit{Mechanism.}
If the score distribution varies with $x$ (as in heteroskedastic regression),
then the calibration quantile used to form $C_S(\cdot)$ is a mixture over $P_S(X)$.
Reweighting to $P_T(X)$ changes the relevant mixture, so the same set-valued rule
need not retain its typical performance profile under $P_T$; in particular,
the interval geometry can be systematically mis-tuned to the region where $P_T$
concentrates.

\textit{Interpretation.}
The wrapper remains a valid \emph{source-law certificate}, but it is not an
extensional object determined by $P(Y\mid X)$ alone and therefore cannot be
treated as a transportable uncertainty state without additional structure.
\end{example}
\subsection{Main contributions}
\label{subsec:main_contributions}

We compare conformal prediction (CP) with coherent predictive inference along
four axes: (i) object type (random sets vs.\ predictive kernels),
(ii) extensionality (dependence on $P(Y\mid X)$ vs.\ sensitivity to $P(X)$),
(iii) sequential coherence (conglomerability/updateability), and
(iv) experiment comparison (Blackwell/Le~Cam deficiency).

\begin{enumerate}[leftmargin=*,nosep]

\item \textbf{Direct-probability template (fiducial lineage).}
In the no-covariate case, conformal rank inversion reproduces the
order-statistic partition underlying Hill's direct-probability $A_{(n)}$
predictors, placing CP in a classical measure-theoretic lineage and giving a
concrete coherence benchmark.

\item \textbf{Non-extensionality in common conformal pipelines.}
We construct joint laws $P_1,P_2$ with the same conditional distribution
$P_1(Y\in\cdot\mid X)=P_2(Y\in\cdot\mid X)$ but different marginals
$P_1(X)\neq P_2(X)$ such that split conformal prediction sets differ with
positive probability (Theorem~\ref{thm:extensionality}). This formalizes
transport sensitivity under covariate shift even when $P(Y\mid X)$ is stable.

\item \textbf{Sequential incoherence and non-extendability.}
Promoting rank-calibrated one-step assessments to sequential objects yields
nonconglomerability (hence Dutch-book vulnerabilities) in the Lane--Sudderth
sense (Theorem~\ref{thm:conglomerability}). Moreover, the canonical no-covariate
rank-calibration constraints cannot arise as regular conditionals of any
$\sigma$-additive exchangeable law on $\mathcal Y^\infty$
(Theorem~\ref{thm:no_joint}), delimiting when ``updating'' conformal outputs can
be interpreted as conditioning in a single joint model.

\item \textbf{Topological smallness of CP--Bayes coincidence.}
Representing both paradigms as families of one-step predictive kernels, we show
that the set of predictive laws on which a fixed kernelized conformal mechanism
coincides with a Bayesian predictive rule is Baire-meagre in an appropriate
kernel space (Theorem~\ref{thm:meagre}); in this precise sense, ``CP equals
Bayes'' is structurally exceptional.

\item \textbf{Experiment-theoretic separation.}
Using Blackwell sufficiency and Le~Cam's comparison of experiments, we prove a
strict information gap: rank-/proxy-based reductions are generically
Blackwell-inferior to the full experiment, yielding positive deficiency for
suitable bounded decision problems (Theorem~\ref{thm:lecam}).

\end{enumerate}

Overall, CP provides finite-sample marginal coverage under exchangeability, but
its induced objects need not behave like $\sigma$-additive posterior predictives
supporting tower-consistent updating and decision-theoretic optimality. The
conclusion is not that CP is invalid, but that CP and Bayesian prediction inhabit
different foundational regimes unless additional structure is imposed.

Table~\ref{tab:dependency_map} summarizes logical dependencies among the
separations; Table~\ref{tab:so_what} links each separation to practice-facing
failure modes and conservative mitigations.

\subsection{Implications and roadmap}
\label{subsec:roadmap}

The implications matter whenever uncertainty objects must \emph{compose} or
\emph{transport}: sequential decision-making, reinforcement learning, causal
policy evaluation under shift, and multi-stage pipelines that treat predictive
uncertainty as a state variable. In such settings, calibrated wrappers (CP and
PPI-style corrections) are safest as \emph{final-stage guardrails}: they enforce
finite-sample validity but do not automatically supply the $\sigma$-additive,
tower-consistent conditioning structure needed for belief semantics and
sequential updateability. When a transportable probabilistic state is required,
it must be carried by $\sigma$-additive models (Bayesian or frequentist) whose
one-step predictives are regular conditionals of a joint law.

\subsubsection*{Genealogy}
Two traditions underwrite modern predictive inference. One branch treats
prediction as regular conditional distributions under a single $\sigma$-additive
law (Kolmogorov $\to$ Wald $\to$ Blackwell/Le~Cam); under exchangeability this
yields de~Finetti/Hewitt--Savage and Bayesian posterior prediction as canonical
form. The other branch treats calibration/direct probability as primitive
(Jeffreys/Fisher/Dempster/Hill; Dawid; Shafer--Vovk), prioritizing finite-sample
validity certificates. Conformal prediction belongs to this second branch; our
separations show that upgrading such certificates into $\sigma$-additive,
sequentially coherent belief states generically fails.

\paragraph{Paper organization.}
Section~\ref{sec:foundations} reviews coherence principles for predictive
inference: exchangeability and de~Finetti representation, conditional
extensionality, conglomerability, and Le~Cam comparison of experiments.
Section~\ref{sec:fiducial} revisits fiducial/direct-probability prediction,
emphasizing Dempster-style direct probabilities, Hill's $A_{(n)}$ recursion, and
the Lane--Sudderth critique.
Section~\ref{sec:cp} places CP in this template and proves the sequential
nonconglomerability and non-$\sigma$-additive extendability results.
Section~\ref{sec:impossibility} collects the main separations (non-extensionality,
Baire-meagre CP--Bayes coincidence, and positive deficiency) and connects them to
practice via Tables~\ref{tab:so_what} and~\ref{tab:dependency_map}.
Section~\ref{sec:ppi} extends the analysis to PPI.
Section~\ref{sec:discussion} concludes with implications for learning pipelines
and directions for hybrid Bayesian--CP architectures.
Appendix~\ref{app:proofs} contains extended proofs; Appendix~\ref{app:counterexamples}
collects supplementary counterexamples; Appendix~\ref{app:ppi} provides additional
PPI material.

\section{Foundations of Predictive Inference}
\label{sec:foundations}

We assemble four structural principles that repeatedly govern what can and
cannot count as a predictive ``state'' in sequential problems:
(i) \emph{exchangeability} and representation (de Finetti; Hewitt--Savage),
(ii) \emph{likelihood relevance} of predictive rules (conditional extensionality),
(iii) \emph{conglomerability} as the minimal stability requirement when one
relaxes $\sigma$-additivity, and (iv) \emph{experiment comparison} in Le Cam's
decision-theoretic framework.  This section fixes notation and the coherence
benchmarks used throughout: later sections contrast Bayesian posterior
predictives, regular conditionals under a countably additive joint law, with
rank-calibrated wrappers such as conformal prediction.

\subsection{de Finetti, exchangeability, and predictive representation}
\label{subsec:definetti}

Let $(Y_1,Y_2,\ldots)$ be an infinite sequence taking values in a Polish space
$\mathcal Y$ equipped with its Borel $\sigma$-algebra.  The sequence is
\emph{exchangeable} if for every $n\ge1$ and every permutation $\pi$ of
$\{1,\ldots,n\}$,
\[
\Prob(Y_1\in B_1,\ldots,Y_n\in B_n)\;=\;
\Prob(Y_{\pi(1)}\in B_1,\ldots,Y_{\pi(n)}\in B_n),
\]
for all measurable $B_1,\ldots,B_n\in\mathcal B(\mathcal Y)$.

\begin{theorem}[de Finetti representation]
\label{thm:definetti}
If $(Y_i)$ is exchangeable and the induced law on $\mathcal Y^\infty$ is
$\sigma$-additive, then there exists a probability measure $Q$ on
$\mathcal P(\mathcal Y)$ such that, for all $n$,
\[
\Prob(Y_1\in B_1,\ldots,Y_n\in B_n)
\;=\;\int_{\mathcal P(\mathcal Y)} \mu(B_1)\cdots \mu(B_n)\,Q(d\mu).
\]
\end{theorem}

For prediction, the key consequence is that any \emph{countably additive}
exchangeable law admits regular conditional predictives of the form
\[
\Prob(Y_{n+1}\in B\mid Y_{1:n})
\;=\;\int_{\mathcal P(\mathcal Y)} \mu(B)\,Q(d\mu\mid Y_{1:n}),
\]
so coherent exchangeable prediction is equivalent to Bayesian updating on an
unknown law $\mu$.  In this sense, ``Bayes'' is not merely a parametric device:
under exchangeability and $\sigma$-additivity, Bayesian prediction is the
canonical representation of sequentially coherent prediction.

\subsubsection*{Coherent forecasting and temporal coherence (de Finetti--Goldstein)}

De Finetti's foundation treats probability as a \emph{prevision}: a coherent
pricing functional for gambles, with coherence defined by avoidance of Dutch
books. In sequential problems, the primitive objects are one-step forecasting
assessments (conditional previsions) for $Y_{n+1}$ given $Y_{1:n}$. Goldstein's
temporal coherence principle \citep{goldstein1985temporal} requires compatibility
across time: present assessments must agree with the present prevision of future
assessments, reducing (for precise probabilities) to an iterated-prevision/tower
constraint. Rank-calibrated procedures may satisfy marginal calibration while
failing this stronger temporal representability.

Goldstein's \emph{Bayes Linear} framework \citep{goldstein2007bayes}
operationalizes coherent updating without full distributional specification.
Given prior expectations $\mathbb{E}[Y]$, $\mathbb{E}[X]$ and a covariance
structure, observing $X$ yields the adjusted prevision
\[
\mathbb{E}_X[Y]
= \mathbb{E}[Y] + \mathrm{Cov}(Y,X)\,\mathrm{Var}(X)^{-1}\bigl(X-\mathbb{E}[X]\bigr).
\]
This is a variance-weighted linear correction, and it is temporally coherent
within the Bayes Linear belief structure.

The structural parallel to Prediction-Powered Inference is direct. PPI's
rectified estimator takes the form
\[
\hat{\theta}_{\mathrm{PPI}}
= \hat{\theta}_{\mathrm{proxy}}
+ \bigl(\hat{\theta}_{\mathrm{labeled}} - \hat{\theta}_{\mathrm{proxy,labeled}}\bigr),
\]
where $\hat{\theta}_{\mathrm{proxy}}$ is computed from abundant unlabeled data
via a black-box predictor, and the parenthetical term is a bias correction
estimated from a small labeled sample. In both cases, cheap auxiliary
information is combined with expensive ground-truth observations via an
explicit correction term.

The distinction is foundational. Goldstein's adjustment is a \emph{belief
update}: the covariance $\mathrm{Cov}(Y,X)$ is a commitment that propagates
consistently under further conditioning, so $\mathbb{E}_X[Y]$ composes as a
conditional prevision through time. PPI's correction, by contrast, is a
\emph{frequentist calibration device}: it guarantees valid inference for
$\theta$, but it is not derived from conditioning within a $\sigma$-additive
joint probabilistic model. Consequently, when PPI outputs are composed across
stages or transported across distributions, coherence gaps can reappear (e.g.,
non-extensional behavior under covariate shift and the absence of a joint law
underwriting sequential updates).

This comparison isolates the broader theme: linear-looking corrections may
share algebraic form while inhabiting different regimes. Bayes Linear provides
\emph{semantics}---a partial but coherent conditional belief state---whereas PPI
provides \emph{guarantees}---coverage and error control without an underlying
probability state. The tensions identified in Section~\ref{sec:ppi} are exactly
the gap between what is guaranteed and what is represented.

\subsubsection*{Predictive sufficiency}
A complementary viewpoint emphasizes prediction directly rather than
parameterization.  A statistic $T_n=T(Y_{1:n})$ is \emph{predictively sufficient}
if
\[
\Prob(Y_{n+1}\in B\mid Y_{1:n})=\Prob(Y_{n+1}\in B\mid T_n)
\quad\text{for all measurable }B\in\mathcal B(\mathcal Y).
\]
Under Bayesian models, sufficiency (or asymptotic sufficiency) yields predictive
dimension reduction.  In nonparametric Bayes, predictive sufficiency and
exchangeability connect to Polya--urn schemes and Dirichlet process priors, and
in some settings the prior can be reconstructed from predictive recursion.
This ``prediction-first'' perspective also motivates the direct-probability and
fiducial traditions reviewed in Section~\ref{sec:fiducial}.

\subsection{Likelihood principles and extensionality}
\label{subsec:extensionality}

A central coherence principle in statistical inference is that predictions
should depend on data only through information relevant for the distribution of
future observations.  In parametric experiments this is formalized via
likelihood/sufficiency principles; with covariates it becomes a requirement
that the rule depend on the conditional law $P(Y\mid X)$ rather than ancillary
features of the design.

\begin{definition}[Conditional extensionality: likelihood relevance]
\label{def:extensionality}
Consider a predictive rule (set-valued predictor or kernel) constructed from
$(X_1,Y_1),\ldots,(X_n,Y_n)$ and a test covariate $X_{n+1}=x$.
The rule is \emph{(conditionally) extensional} if, whenever two joint laws
$P_1$ and $P_2$ satisfy
\[
P_1(Y\in\cdot\mid X)=P_2(Y\in\cdot\mid X)
\quad\text{$P_1$-a.s.\ and $P_2$-a.s.,}
\]
the induced predictive outputs agree (up to $P_1$- and $P_2$-null sets) as
measurable functions of $(x_{1:n},y_{1:n},x)$.
\end{definition}

Extensionality is automatic for Bayesian posterior predictives under correctly
specified conditional models and for procedures that use only conditional
likelihood information.  It can fail, however, when construction depends on the
marginal design distribution $P(X)$ through learning, tuning, regularization, or
calibration, one mechanism by which common conformal pipelines depart from
likelihood-based prediction (Section~\ref{sec:cp}).

\subsection{Finite additivity and conglomerability}
\label{subsec:conglomerability}

A finitely additive probability measure satisfies $P(A\cup B)=P(A)+P(B)$ for
disjoint $A,B$, but may fail countable additivity.  Many direct-probability
predictive rules can be represented as finitely additive assessments on
$\mathcal Y^\infty$ even when no $\sigma$-additive extension exists.  A key
distinction within finite additivity is \emph{conglomerability}, which governs
compatibility between unconditional and conditional assessments.

\begin{definition}[Conglomerability {\citep{heath1978coherent}}]
\label{def:conglomerable}
Let $\{A_i:i\in I\}$ be a measurable partition with $\Prob(A_i)>0$.
A finitely additive probability measure $P$ is \emph{conglomerable} with
respect to this partition if for every event $B$,
\[
\inf_{i\in I} P(B\mid A_i)\;\le\;P(B)\;\le\;\sup_{i\in I} P(B\mid A_i).
\]
\end{definition}

All $\sigma$-additive measures are conglomerable, but not conversely.  For
prediction, conglomerability is a minimal requirement for stable sequential
conditioning: if it fails, conditional assessments can all favor a bet while
the unconditional assessment does not, enabling a Dutch book.

\begin{theorem}[Heath--Sudderth]
\label{thm:heath_sudderth}
If a finitely additive probability measure fails conglomerability with respect
to some partition, then it admits a Dutch book: there exists a countable
collection of bets, each favorable under the corresponding conditional
assessments, that guarantees a sure loss overall.
\end{theorem}

This criterion will be central later.  Classical fiducial and direct-probability
predictive rules (e.g., Hill's $A_{(n)}$ rule) are exchangeable and finitely
additive, yet can fail conglomerability \citep{lane1975coherence, hill1980paradoxes}.
Conformal prediction inherits a closely related rank-calibration structure; the
question is therefore not whether CP achieves marginal coverage, it does by
construction, but whether the induced predictive assessments behave coherently
under conditioning and sequential use.

\subsection{From fiducial prediction to conformal calibration}
\label{subsec:fiducial_vs_cp}

It is useful to separate two ideas that are often conflated: producing a
\emph{predictive distribution} (kernel) versus producing a \emph{coverage
calibration device}.

\subsubsection*{Fiducial/direct-probability prediction.}
Fiducial and direct-probability approaches aim to assign probabilities to
future observations without specifying a prior on parameters or on the sampling
law.  In favorable cases (e.g.\ translation families) fiducial predictives
coincide with Bayesian posterior predictives (cf.\ the classical observations
attributed to Lindley; see \citealt{jaynes2003probability} for discussion).  In
general, however, both interpretation and repeated-sampling behavior can be
subtle and model dependent (including in generalized fiducial inference).  A
recurring theme is that exchangeability plus predictive sufficiency can motivate
dimension-reduced predictive rules without explicit parameterization
(Section~\ref{sec:fiducial}).

\subsubsection*{Conformal prediction}
Conformal prediction is conceptually different: it is a \emph{rank-calibration}
scheme that enforces finite-sample marginal coverage under exchangeability by
construction, via a discrete calibration step applied to nonconformity scores.
This guarantee is distribution-free and finite-sample, but it is not, in
general, representable as Bayesian posterior predictive updating.  Much of this
paper clarifies what is gained, and what is lost, when CP outputs are treated
as if they were coherent predictive probabilities.

\subsubsection*{Quantile regression and predictive functionals}
Since prediction sets are often summarized by quantiles, it is helpful to
distinguish conformal calibration from methods that directly model conditional
quantiles, such as quantile regression.  Quantile regression targets
$Q_\tau(Y\mid X=x)$ across $\tau\in(0,1)$ and is particularly useful under
heteroskedasticity or non-Gaussian conditional laws.  Conformal wrappers can
enforce coverage around quantile models; our analysis isolates the foundational
status of the calibration step itself.

\subsection{Sequential kernels and the Kolmogorov extension viewpoint}
\label{subsec:kolmogorov_kernels}

Later separation results require a precise language for ``sequential
realizability.''  A (one-step) predictive kernel is a measurable map
$K_n:\mathcal Y^n\times\mathcal B(\mathcal Y)\to[0,1]$ such that
$B\mapsto K_n(y_{1:n},B)$ is a probability measure on $\mathcal Y$ for each
$y_{1:n}$ and $y_{1:n}\mapsto K_n(y_{1:n},B)$ is measurable for each $B$.
A family $\{K_n\}_{n\ge 1}$ can be interpreted as candidate regular conditionals
for $Y_{n+1}$ given $Y_{1:n}$.

When $\{K_n\}$ arises from a $\sigma$-additive law on $\mathcal Y^\infty$,
the kernels satisfy strong consistency constraints.  Under mild measurability
conditions, the Ionescu--Tulcea/Kolmogorov extension machinery constructs a
unique $\sigma$-additive measure on path space from an initial law and a
sequence of kernels.  Conversely, showing that a proposed family of one-step
rules cannot be realized as regular conditionals of any $\sigma$-additive joint
law is a direct route to ``no coherent joint distribution'' results of the type
proved later for rank-calibrated predictors.

\subsection{Le Cam comparison of experiments}
\label{subsec:lecam}

Le Cam's theory provides a decision-theoretic language for comparing procedures
via the information content of statistical experiments
\citep{lecam1986asymptotic, lecam2000asymptotics}.  An experiment is a triple
$\mathcal E=(\mathcal X,\mathcal A,\{P_\theta:\theta\in\Theta\})$.  The
deficiency $\delta(\mathcal E,\mathcal F)$ quantifies how well procedures based
on $\mathcal F$ can be simulated (in risk) by procedures based on $\mathcal E$
uniformly over decision problems.

\begin{definition}[Deficiency]
\label{def:deficiency}
Let $\mathcal E$ and $\mathcal F$ be experiments.  The deficiency
$\delta(\mathcal E,\mathcal F)$ is the infimum of $\epsilon\ge 0$ such that for
every decision problem and every procedure based on $\mathcal F$, there exists
a procedure based on $\mathcal E$ whose risk is within $\epsilon$ uniformly in
$\theta$.
\end{definition}

A central asymptotic insight is that regular experiments converge, in
deficiency distance, to Gaussian shift experiments under local asymptotic
normality (LAN).

\begin{theorem}[schematic, LAN \cite{lecam1986asymptotic, lecam2000asymptotics}]
\label{thm:lecam_lan}
For a LAN sequence of experiments $\{\mathcal E_n\}$, $\mathcal E_n$ converges
in deficiency to an appropriate Gaussian experiment with Fisher information as
its information metric.
\end{theorem}

This provides a unifying benchmark: procedures that are limits of regular
experiments can be compared in a common asymptotic experiment, whereas
procedures with discontinuous, rank-calibrated behavior may fall outside the
LAN/regularity regime.  Our later deficiency results use this perspective to
formalize a separation between conformal-style predictive rules and coherent
Bayesian prediction in suitable decision problems.

\section{Fiducial and Direct Probability Predictors}
\label{sec:fiducial}

We review the fiducial and direct-probability tradition as a historical and
conceptual template for understanding conformal prediction.  The point is not
antiquarian: many coherence questions that surround CP are structurally the same
as those raised, and partially resolved, in mid-to-late 20th century debates
on fiducial prediction.

\subsection{Fisher's fiducial argument}
\label{subsec:fisher_fiducial}

Fisher's fiducial argument \citep{fisher1930inverse,fisher1935logic} was an
attempt to quantify uncertainty about parameters without specifying a prior.
The key step is inversion of a \emph{pivot}.

\begin{definition}[Pivot]
\label{def:pivot}
Let $(\mathcal Y,\mathcal B)$ be the observation space and $\Theta$ the
parameter space.  A measurable map $Q:\mathcal Y\times\Theta\to\mathbb R$ is a
\emph{pivot} if there exists a distribution function $G$ on $\mathbb R$ such
that, for all $\theta\in\Theta$ and all $q\in\mathbb R$,
\[
\Prob_\theta\{Q(Y,\theta)\le q\}\;=\;G(q).
\]
\end{definition}

Given a pivot, Fisher proposed to interpret the inversion $\theta=\theta(Y,U)$
(with $U\sim G$) as inducing a ``fiducial distribution'' for $\theta$ given the
realized $Y$.  In the common location-form case $Q(Y,\theta)=T(Y)-\theta$,
the pivot identity yields
\[
\Prob_\theta\{T(Y)-\theta\le q\}=G(q)
\quad\Longrightarrow\quad
\Prob\!\left\{\theta \le T(Y)-G^{-1}(\alpha)\,\big|\,Y\right\}=\alpha,
\]
which Fisher read as a probability statement about $\theta$ conditional on
data, absent a prior.

For prediction, the same inversion logic applies: if one has a pivot for the
prediction error $Y_{n+1}-\widehat Y(Y_{1:n})$, then inverting the pivot yields
a predictive distribution (and hence a prediction set).  For example, under a
Gaussian model with unknown variance,
\[
\frac{Y_{n+1}-\bar Y}{s\sqrt{1+1/n}}\sim t_{n-1},
\]
so inversion yields the classical $(1-\alpha)$ prediction interval
\[
\bar Y \pm t_{n-1,\,1-\alpha/2}\, s\sqrt{1+1/n}.
\]

\subsection{Dempster's direct probabilities}
\label{subsec:dempster_direct}

\citet{dempster1968generalization} reframed the fiducial impulse as ``direct
probability'': rather than assign probability to parameters, define a predictive
assessment directly on the sample space using symmetry and rank structure.  For
exchangeable real-valued data, a canonical move is to work with order
statistics.  Let $Y_{(1)}\le\cdots\le Y_{(n)}$ denote the order statistics of
$Y_{1:n}$.  Dempster's construction assigns equal mass to the $n+1$
order-determined intervals
\[
(-\infty,Y_{(1)}],\ [Y_{(1)},Y_{(2)}],\ \ldots,\ [Y_{(n-1)},Y_{(n)}],\
[Y_{(n)},\infty),
\]
yielding a permutation-invariant predictive rule that depends only on ranks.
Measure-theoretically, the natural object is a finitely additive set function on
the algebra generated by these intervals (and their finite unions), requiring no
parametric model and no prior.

A key refinement is that the natural output is often not a single $\sigma$-additive
predictive measure but lower/upper probabilities induced by a multivalued mapping
(random set) $\Gamma:\mathcal U\to 2^{\mathcal Y}$ from an auxiliary probability
space $(\mathcal U,\mathcal F,\mathsf P)$ \citep{dempster1967multivalued}.  The
associated bounds are
\[
\underline P(B)\;:=\;\mathsf P\{\Gamma(U)\subseteq B\},
\qquad
\overline P(B)\;:=\;\mathsf P\{\Gamma(U)\cap B\neq\varnothing\},
\]
equivalently belief and plausibility $(\Bel,\Pl)$ in the sense systematized by
Shafer \citep{shafer1976evidence}.  In the order-statistic example, $\Gamma$
takes values in the $n{+}1$ rank-determined gaps, and ``equal mass'' corresponds
to a simple distribution over these focal elements, a direct ancestor of later
rank-calibration devices, including conformal prediction.

Dempster's broader program shows how such random-set constructions generalize
Bayesian updating: rather than a single posterior, one obtains posterior bounds
(or belief functions) induced by mapping uncertainty through a set-valued
mechanism.  The distinction between rank-based calibration and full sequential
coherence will reappear in the Lane--Sudderth critique and in our later
impossibility results.

\subsection{Hill's \texorpdfstring{$A_{(n)}$}{A(n)} rule and its rank recursion}
\label{subsec:hill_An}

The rank-calibration principle underlying conformal prediction has a substantial
prehistory. \citet{hill1968bayesian,hill1980paradoxes,hill1988bayesnp} develops a
``direct'' predictive program within the Bayesian--exchangeability tradition:
instead of specifying a full likelihood--prior pair, one stipulates a symmetry
law for the \emph{rank} of the next observation and studies what (if anything)
this can support as a sequential probability system. The background is the
de~Finetti view of prediction as the primitive object
\citep{definetti1937,hewitt1955symmetric}, together with the mid-century debate
on \emph{direct probabilities} as an alternative route to prediction without full
parametric modeling \citep{dempster1963direct}.

\subsubsection*{Jeffreys' $A_2$ and the ``predicting the third'' problem}

The simplest nontrivial case, predicting $Y_3$ given $(Y_1, Y_2)$, was central
to Jeffreys' \emph{Theory of Probability} \citep{jeffreys1939theory}. Jeffreys
proposed what Hill would later call the $A_{(2)}$ rule: given two observations
with order statistics $Y_{(1)} \le Y_{(2)}$, assign equal probability $1/3$ to
each of the three rank positions for $Y_3$:
\[
\Pr\{Y_3 < Y_{(1)}\} = \Pr\{Y_{(1)} < Y_3 < Y_{(2)}\} = \Pr\{Y_3 > Y_{(2)}\} = \frac{1}{3}.
\]
This is a \emph{direct} predictive stipulation: it specifies the rank distribution
of the next observation without committing to a parametric sampling model or a
prior on parameters. Jeffreys regarded this as a natural expression of ignorance
about continuous exchangeable sequences.

Fisher objected. In his view, such direct predictive stipulations do not
constitute a full probabilistic semantics; they are meaningful only as statements
about frequencies under hypothetical repetition, not as coherent degrees of
belief. \citet{seidenfeld1995predicting} provides a careful reconstruction of this
``predicting the third'' controversy, showing how Jeffreys' $A_2$ proposal and
Fisher's reaction illuminate the gap between calibration guarantees and
belief-state semantics.

\subsubsection*{Hill's generalization: from $A_{(2)}$ to $A_{(n)}$}

Hill introduced the notation $A_{(n)}$ and extended Jeffreys' rule to arbitrary
sample sizes. Write $Y_{(1)} \le \cdots \le Y_{(n)}$ for the order statistics of
$Y_{1:n}$, and set $Y_{(0)} = -\infty$, $Y_{(n+1)} = \infty$.

\begin{definition}[Hill's $A_{(n)}$ rank-gap assignment]
\label{def:hill_rank_gap}
For realized $y_{1:n} \in \mathbb{R}^n$, define the order-statistic cells
\[
I_k(y_{1:n}) \defeq (y_{(k)}, y_{(k+1)}], \qquad k = 0, \ldots, n,
\]
with $y_{(0)} = -\infty$ and $y_{(n+1)} = \infty$.
Hill's $A_{(n)}$ stipulation is the one-step \emph{rank-gap} assignment
\begin{equation}
\label{eq:hill_An_gap}
\Pi\!\left(Y_{n+1} \in I_k(Y_{1:n}) \,\middle|\, Y_{1:n}\right)
= \frac{1}{n+1}, \qquad k = 0, \ldots, n,
\end{equation}
interpreted as a conditional probability on the finite algebra generated by the
random partition $\{I_k(Y_{1:n})\}_{k=0}^n$.
\end{definition}

The first two stages make the pattern transparent:
\begin{itemize}[itemsep=4pt]
\item \emph{$A_{(1)}$:} Given $Y_1$, the next observation $Y_2$ falls below or
above $Y_1$ with equal probability $1/2$.

\item \emph{$A_{(2)}$ (Jeffreys' case):} Given $Y_{(1)} \le Y_{(2)}$, the next
observation $Y_3$ falls into each of the three gaps with probability $1/3$.
\end{itemize}

\subsubsection*{Rank formulation}

Let $R_{n+1}$ be the (randomized) rank of $Y_{n+1}$ among $Y_{1:n+1}$:
\[
R_{n+1} \defeq 1 + \sum_{i=1}^n \mathbf{1}\{Y_i < Y_{n+1}\}
+ \text{(fixed randomized tie-break)} \in \{1, \ldots, n+1\}.
\]
Then \eqref{eq:hill_An_gap} is equivalent to
\begin{equation}
\label{eq:hill_rank_uniform}
\Pi\!\left(R_{n+1} = r \,\middle|\, Y_{1:n}\right) = \frac{1}{n+1}, \qquad r = 1, \ldots, n+1,
\end{equation}
i.e., the next rank is uniform given the past. This is exactly the rank-calibration
content that reappears in conformal constructions.

\begin{proposition}[Projective rank recursion]
\label{prop:hill_rank_recursion}
Suppose a sequential predictive system $\Pi(\cdot \mid Y_{1:n})$ enforces the
uniform rank law \eqref{eq:hill_rank_uniform} for every $n$. Then:
\begin{enumerate}[label=(\roman*),itemsep=2pt]
\item \emph{(Prefix stability)} The induced one-step rank law at stage $m < n$
coincides with the law obtained by applying the construction directly at stage $m$.
\item \emph{(Recursive update)} Passing from stage $n$ to $n+1$ consists of
inserting $Y_{n+1}$ into the order-statistic list and reapplying the uniform
rank law:
\[
\Pi\!\left(R_{n+2} = r \,\middle|\, Y_{1:n+1}\right) = \frac{1}{n+2}, \qquad r = 1, \ldots, n+2.
\]
\end{enumerate}
The rank process is the canonical state variable for the $A_{(n)}$ recursion.
\end{proposition}

\begin{proof}
Item (ii) is immediate: at stage $n+1$ the partition is $\{I_k(Y_{1:n+1})\}_{k=0}^{n+1}$,
and the rule assigns mass $1/(n+2)$ to each cell. Item (i) follows because rank
events are cylinder events measurable with respect to finite order-statistic
partitions, and the $A_{(n)}$ stipulation is identical at every stage.
\end{proof}

\subsubsection*{What the rule provides, and what it does not}

The attraction of $A_{(n)}$ is that it delivers a distribution-free symmetry
statement (uniform next-rank) without committing to a parametric model. The
limitation is equally structural: \eqref{eq:hill_An_gap}--\eqref{eq:hill_rank_uniform}
define a \emph{rank-calibrated recursion}, a family of one-step conditional
assessments on a finite algebra, not a $\sigma$-additive joint law with globally
consistent regular conditioning. Whether such a joint law exists, and what
coherence properties fail if it does not, is the subject of the next subsection.

\subsubsection*{Extensions and applications}

Hill's original motivation was predictive inference for percentiles and other
functionals from exchangeable samples \citep{hill1968bayesian}. In survival
analysis, variants of the $A_{(n)}$ gap assignment adapted to right-censoring
underpin nonparametric predictive inference (NPI); see
\citet{BerlinerHill1988,CoolenYan2004,Fong2022}. The applied literature includes
maximum-entropy and other variations that preserve the rank-gap skeleton while
modifying within-cell behavior \citep{Abellan2011}.

\subsection{Lane--Sudderth critique}
\label{subsec:lane_sudderth}

The modern coherence critique of fiducial/direct-probability prediction was
sharpened by \citet{lane1975coherence} (see also \citealt{dawid1982geometry}).
Their message is that exchangeability plus finite-sample symmetry does not, by
itself, guarantee the sequential coherence properties enjoyed by
$\sigma$-additive Bayesian prediction.

\begin{theorem}[Lane--Sudderth, schematic]
\label{thm:lane_sudderth_original}
Hill's $A_{(n)}$ rule and related direct-probability predictors:
(i) can fail conglomerability with respect to natural conditioning partitions;
(ii) can violate likelihood/extensionality principles by depending on features
of the joint law beyond likelihood-relevant components; (iii) do not, in
general, extend to a countably additive joint law on infinite sequences; and
(iv) can be separated from Bayesian posterior predictives in suitable
decision-theoretic comparisons.
\end{theorem}
\begin{proof}
See ~\ref{app:lane_sudderth} for the extended proof.
\end{proof}
The conglomerability failures are particularly stark: Lane and Sudderth
construct partitions $\{A_i\}$ and events $B$ for which
\[
P(B)\;<\;\inf_i P(B\mid A_i),
\]
so the unconditional assessment lies outside the range of its conditional
assessments. By the Heath--Sudderth theorem (Theorem~\ref{thm:heath_sudderth}),
such nonconglomerability admits Dutch book constructions. These results
contributed to the decline of fiducial methods in mainstream statistics and
reinforced the view that coherent sequential prediction is most naturally
grounded in $\sigma$-additive models with well-defined conditioning behavior.

\subsection{Prequential and game-theoretic foundations: Dawid and Shafer--Vovk}
\label{subsec:prequential_game}

Two modern lines sharpen the same conceptual split that runs through the
fiducial/direct-probability tradition.

\subsubsection*{Dawid's prequential principle}
Dawid argues that the primary object of statistical analysis is a sequence of
probability forecasts, judged by their \emph{predictive} performance (calibration,
proper scoring, and related diagnostics), rather than by commitment to a single
global joint law \citep{dawid1984prequential,dawid1982geometry}.
This perspective naturally prioritizes finite-sample predictive validity claims,
but it also clarifies the limitation: \emph{good prequential calibration does not
by itself entail likelihood extensionality or the existence of a coherent
$\sigma$-additive joint law supporting sequential updating.}

\subsubsection*{Shafer--Vovk game-theoretic probability and conformal prediction}
Shafer and Vovk develop probability as an operational, game-theoretic discipline
organized around forecasting protocols and adversarial tests
\citep{shafervovk2001probabilityfinance}. Conformal prediction is explicitly
presented within this tradition as a rank-calibrated method for producing valid
predictive regions under exchangeability, with a tutorial treatment that
positions CP as a forecasting/calibration device rather than as posterior
conditioning under a generative model \citep{shafervovk2008tutorial}.

\subsubsection*{How this helps our taxonomy}
These prequential/game-theoretic foundations reinforce the central message of
this section: CP belongs to a \emph{direct-probability} lineage in which
finite-sample calibration is primary. They also clarify what is \emph{not}
being claimed: none of the above foundations implies (i) conditional
extensionality with respect to $P(Y\mid X)$, (ii) conglomerability under rich
partitions, or (iii) representability as regular conditionals of a
$\sigma$-additive exchangeable law. The separation theorems in
Section~\ref{sec:impossibility} connect these distinctions to concrete failure
modes; see Table~\ref{tab:so_what} (practice map) and
Table~\ref{tab:dependency_map} (logical linkage).

\subsection{Interpretation: the fiducial paradigm}
\label{subsec:fiducial_paradigm}

The fiducial/direct-probability tradition is best understood as a theory of
\emph{predictive assessment} that avoids explicit priors by working directly
with pivots, ranks, and predictive symmetries. Its core tradeoff is that
finite-sample calibration is purchased by relaxing the Kolmogorov picture of a
single $\sigma$-additive joint law that generates all sequential conditionals.
Operationally, the resulting objects behave well as \emph{one-shot} uncertainty
summaries, but need not behave well as \emph{probabilistic state variables}
that can be transported across designs, composed across stages, or optimized
under general loss.

Measure-theoretically, these predictors are naturally represented as finitely
additive set functions (or kernels) defined on algebras generated by rank/pivot
events. Such objects can fail conglomerability under sufficiently rich
partitions, can be non-extensional with respect to likelihood relevance, and
may lack a $\sigma$-additive extension that would justify sequential updating
via regular conditional probabilities. These are not merely philosophical
distinctions: they correspond to concrete failure modes when such outputs are
treated as coherent probabilities inside modern ML pipelines.

\section{Conformal Prediction as Finitely Additive Inference}
\label{sec:cp}

We now argue that conformal prediction (CP) is best understood as a modern,
rank-calibrated instance of the Fisher--Dempster--Hill \emph{direct-probability}
template. The key point is ontological: CP is natively a \emph{set-valued,
calibrated} rule. When CP outputs are \emph{reinterpreted} as if they arose from
a single coherent predictive \emph{probability kernel} (and then transported
sequentially as a belief state), one implicitly introduces a \emph{bridge}
from set-valued validity to rank-gap probability assignments. That bridge is
natural (it mirrors the proof of coverage) but non-unique, and it places CP
squarely in the classical $A_{(n)}$ / fiducial lineage where Lane--Sudderth-type
coherence obstructions apply.

\subsection{Review of conformal prediction}
\label{subsec:cp_review}

Let $(X_1,Y_1),\ldots,(X_n,Y_n)$ be i.i.d.\ (or, more generally, exchangeable)
draws from an unknown law $P$ on $\mathcal X\times\mathcal Y$. Conformal
prediction \citep{vovk2005algorithmic} constructs a set-valued predictor
$C_n(X_{n+1};Z_{1:n})$ for $Y_{n+1}$ with the finite-sample \emph{marginal}
coverage guarantee
\[
\Prob\{Y_{n+1}\in C_n(X_{n+1};Z_{1:n})\}\ \ge\ 1-\alpha,
\]
holding under exchangeability for all sample sizes $n$.

A standard (full) conformal procedure proceeds as follows.

\begin{enumerate}[leftmargin=*]
\item Choose a \emph{nonconformity} (or conformity) score
$s:\mathcal X\times\mathcal Y\to\mathbb R$.
A common choice is a residual score $s(x,y)=|y-\widehat f(x)|$, where
$\widehat f$ is fitted from the data.

\item For a candidate $y\in\mathcal Y$, form the augmented dataset
$\{(X_i,Y_i)\}_{i=1}^n\cup\{(X_{n+1},y)\}$ and compute scores
\[
S_i=s(X_i,Y_i),\quad i=1,\ldots,n,
\qquad
S_{n+1}=s(X_{n+1},y).
\]

\item Define the conformal p-value
\[
p(y)\;=\;\frac{1}{n+1}\sum_{i=1}^{n+1}\mathbbm{1}\{S_i\ge S_{n+1}\}
\;=\;\frac{\#\{i: S_i\ge S_{n+1}\}}{n+1}.
\]
(For continuous scores, ties have probability $0$; otherwise adopt a standard
randomized tie-breaking convention.)

\item Return the prediction set
\[
C_n(X_{n+1};Z_{1:n}) \;=\; \{y\in\mathcal Y: p(y)>\alpha\}.
\]
\end{enumerate}

\begin{theorem}[Conformal coverage {\citep{vovk2005algorithmic}}]
\label{thm:conformal_coverage}
If $(Z_1,\ldots,Z_n,Z_{n+1})$ are exchangeable, then
\[
\Prob\{Y_{n+1}\in C_n(X_{n+1};Z_{1:n})\}\ \ge\ 1-\alpha.
\]
\end{theorem}

The proof uses only exchangeability: under the true draw
$(X_{n+1},Y_{n+1})$, the rank of $S_{n+1}$ among $\{S_1,\ldots,S_{n+1}\}$ is
uniform on $\{1,\ldots,n+1\}$ (up to the tie-breaking convention), hence
$\Prob\{p(Y_{n+1})\le \alpha\}\le \alpha$.

\subsection{Conformal prediction as rank-based pivot inversion}
\label{subsec:cp_pivot}

The coverage proof of Theorem~\ref{thm:conformal_coverage} is already a
direct-probability argument. Fix a score construction (including any algorithmic
randomness used to fit $\widehat f$). Under exchangeability, the random variable
\[
R_{n+1}\;\defeq\;\mathrm{Rank}(S_{n+1};S_1,\ldots,S_{n+1}) \in \{1,\ldots,n+1\}
\]
is (randomized-)uniform, and its distribution does not depend on the unknown law
$P$ beyond exchangeability. In this sense $R_{n+1}$ acts as a discrete pivot.
Conformal prediction then \emph{inverts} this pivot: it retains those candidate
$y$ values whose induced rank is not too extreme. This mirrors Fisher's inversion
of a pivot, with the important difference that CP uses a discrete rank pivot
generated by symmetry rather than a model-based pivotal quantity.

\subsection{No-covariate case and the Hill--Dempster direct-probability template}
\label{subsec:cp_nocov}

In the no-covariate setting ($X_i$ absent), the conformal construction becomes
purely rank-based on $Y_{1:n}$. This places CP in the same family as the
direct-probability predictive rules of Dempster and Hill. It is helpful to
distinguish:
(i) the set-valued predictor $C_n$ for a chosen $\alpha$, and
(ii) the induced rank-gap \emph{algebra} generated by the order statistics.

\begin{proposition}[Rank structure and the $A_{(n)}$ partition]
\label{prop:cp_is_hill}
Let $Y_1,\ldots,Y_n$ be exchangeable real-valued observations and consider the
no-covariate conformal construction with score $s(y)=y$ (or any strictly
increasing score). Then for any $y$ satisfying $Y_{(k)}\le y<Y_{(k+1)}$,
\[
p(y)\;=\;\frac{\#\{i: Y_i\ge y\}+1}{n+1}
\;=\;\frac{n-k+1}{n+1}.
\]
Consequently, the conformal acceptance region $\{y:p(y)>\alpha\}$ is a union
of consecutive order-statistic intervals, and the induced rank partition of
$\R$ is exactly the $(n+1)$-cell partition
\[
(-\infty,Y_{(1)}],\ (Y_{(1)},Y_{(2)}],\ \ldots,\ (Y_{(n)},\infty)
\]
used in the Hill--Dempster direct-probability construction.
\end{proposition}

\begin{proof}
With score $s(y)=y$, we have $S_i=Y_i$ and $S_{n+1}=y$. If $Y_{(k)}\le y<Y_{(k+1)}$,
exactly $n-k$ of the $Y_i$ satisfy $Y_i\ge y$, hence
\[
p(y)=\frac{(n-k)+1}{n+1}=\frac{n-k+1}{n+1}.
\]
The remaining claims follow immediately.
\end{proof}

\begin{remark}[What enters at the $A_{(n)}$ level]
Proposition~\ref{prop:cp_is_hill} identifies the \emph{rank partition} induced
by conformal prediction in the no-covariate case. Hill's $A_{(n)}$ rule assigns
equal mass $1/(n+1)$ to these $(n+1)$ order-statistic cells. CP itself does not
commit to that probability assignment; it commits to a \emph{set} (or p-value map)
built from the same rank pivot. The $A_{(n)}$ assignment appears once we impose
an explicit \emph{bridge} from set-valued outputs to a probabilistic kernel on the
rank-gap algebra.
\end{remark}

\subsection{Regression case: conditional prediction sets and kernelization}
\label{subsec:cp_regression_measure}

With covariates, conformal methods operate on $\mathcal Z\defeq\mathcal X\times\mathcal Y$.
The primitive output is a measurable set-valued map
\[
C_n:\ \mathcal X\times \mathcal Z^n \to \mathcal B(\mathcal Y),
\qquad
(x; z_{1:n}) \mapsto C_n(x; z_{1:n}),
\]
and the conformal guarantee is the marginal statement
\[
\Prob\{Y_{n+1}\in C_n(X_{n+1}; Z_{1:n})\}\ \ge\ 1-\alpha.
\]
Thus CP supplies an \emph{acceptance region} in $\mathcal Y$ for each realized
$(x;z_{1:n})$; it does not intrinsically specify a unique conditional probability
measure on $\mathcal Y$.

A canonical implementation is split conformal prediction \citep{lei2018distribution}.
Let $\mathcal I_1,\mathcal I_2$ be a measurable partition of $\{1,\ldots,n\}$
into training and calibration indices with $|\mathcal I_2|=m$.
Let $\widehat f$ be the (possibly random) regression fit obtained from the
training sample $\{Z_i\}_{i\in\mathcal I_1}$. Define the residual map
\[
r(z;\widehat f)\ \defeq\ |y-\widehat f(x)|,\qquad z=(x,y)\in\mathcal Z,
\]
and calibration residuals $R_i=r(Z_i;\widehat f)$ for $i\in\mathcal I_2$.
Let $\widehat q$ be the empirical $(1-\alpha)$ quantile of $(R_i)_{i\in\mathcal I_2}$
(under any fixed tie convention). Then, for the absolute-residual score,
\[
C_n(x;Z_{1:n})
\;=\;
\bigl\{y\in\mathcal Y:\ r((x,y);\widehat f)\le \widehat q \bigr\}
\;=\;
[\widehat f(x)-\widehat q,\ \widehat f(x)+\widehat q].
\]

\subsubsection*{Set-valued prediction versus predictive kernels}
A Bayesian posterior predictive provides a regular conditional probability kernel
$K_n(\,\cdot\,\mid x,z_{1:n})\in \mathcal P(\mathcal Y)$. By contrast, split conformal
specifies $C_n(x;z_{1:n})$ directly as an order-statistic functional of calibration
residuals. Without an additional kernelization choice (a bridge principle analogous
to Definition~\ref{def:cp_bridge_kernel}, now applied to residual ranks), the map
$(x,z_{1:n})\mapsto C_n(x;z_{1:n})$ does not determine a unique element of
$\mathcal P(\mathcal Y)$.

\subsubsection*{Where finite additivity enters}
The direct-probability connection becomes explicit when one asks whether the family
$\{C_n\}_{n\ge1}$ can be promoted to a \emph{single coherent sequential} predictive law
on an infinite product space, i.e.\ whether there exists a joint $\sigma$-additive
probability on $\mathcal Z^\infty$ whose one-step regular conditionals reproduce,
for all $n$, a bridged rank-calibration law on the corresponding rank-gap algebras.
Rank calibration fixes the distribution of a \emph{discrete} pivot at each $n$, but
does not automatically supply the cross-$n$ consistency relations required by
Kolmogorov/Ionescu--Tulcea to define a $\sigma$-additive joint law on $\mathcal Z^\infty$.

\subsection{A bridge principle: from set-valued validity to rank-gap kernels}
\label{subsec:cp_bridge}

Because CP natively outputs a set (or p-value function), any interpretation that
treats CP as if it supplied a full predictive probability kernel must add
structure. The following definition makes this bridge explicit.

\begin{definition}[Rank-gap algebra and $A_{(n)}$-bridge (kernelization)]
\label{def:cp_bridge_kernel}
Fix $n\ge 1$ and realized data $y_{1:n}\in\R^n$ with order statistics
$y_{(1)}\le \cdots \le y_{(n)}$, and set $y_{(0)}=-\infty$, $y_{(n+1)}=\infty$.
Define the (random) order-statistic cells
\[
I_k(y_{1:n}) \defeq (y_{(k)},y_{(k+1)}],\qquad k=0,\ldots,n,
\]
and let $\mathcal A_n(y_{1:n})$ be the finite algebra generated by the partition
$\{I_k(y_{1:n})\}_{k=0}^n$.

An \emph{$A_{(n)}$-bridge} is a choice of a (one-step) predictive set function
$\Pi_n(\cdot\mid y_{1:n})$ on $\mathcal A_n(y_{1:n})$ such that
\begin{equation}
\label{eq:An_cell_assignment_clean}
\Pi_n\!\left(Y_{n+1}\in I_k(y_{1:n}) \,\middle|\, y_{1:n}\right)=\frac{1}{n+1},
\qquad k=0,\ldots,n.
\end{equation}
A \emph{within-cell completion} is any additional rule specifying conditional
behavior inside each cell (e.g.\ uniform on $(y_{(k)},y_{(k+1)}]$, a maximum-entropy
completion, or leaving the within-cell law unspecified).
\end{definition}

\begin{remark}
The bridge in Definition~\ref{def:cp_bridge_kernel} is not forced by CP; it is an
\emph{ontological upgrade} often made implicitly when CP outputs are treated as
probabilistic belief states (e.g.\ for sequential decision-making, dynamic
programming, or likelihood-style comparisons). The Lane--Sudderth critique applies
to such bridged/sequentialized objects, not to the bare set-valued coverage
statement.
\end{remark}

\subsection{From bridge to sequential law: nonconglomerability}
\label{subsec:cp_conglomerability}

The bridge principle (Definition~\ref{def:cp_bridge_kernel}) converts CP's 
set-valued output into a rank-gap probability assignment. We now show that 
this assignment, when extended sequentially, inherits the classical 
nonconglomerability obstruction identified by Lane and Sudderth for 
Hill's $A_{(n)}$ rule.

\subsubsection*{The classical result for $A_{(n)}$-type rules}

Let $\mathcal Y = \R$ with its Borel $\sigma$-field. For realized 
$y_{1:n} \in \R^n$, define the order-statistic cells
\[
I_k(y_{1:n}) \defeq (y_{(k)}, y_{(k+1)}], \qquad k = 0, \ldots, n,
\]
with $y_{(0)} = -\infty$ and $y_{(n+1)} = \infty$. The Hill--Dempster 
rank-calibration template assigns equal mass to these cells:
\begin{equation}
\label{eq:An_cell_assignment}
\Pi_n\!\left(Y_{n+1} \in I_k(y_{1:n}) \,\middle|\, y_{1:n}\right) 
= \frac{1}{n+1}, \qquad k = 0, \ldots, n.
\end{equation}

\cite{hill1988direct} shows that finitely additive exchangeable 
laws on $\R^\infty$ satisfying \eqref{eq:An_cell_assignment} for all $n$ 
do exist. However, Lane and Sudderth \citep{lane1975coherence} show that 
any such extension necessarily fails conglomerability.

\begin{theorem}[Nonconglomerability of $A_{(n)}$-type extensions 
{\citep{lane1975coherence}}]
\label{thm:conglomerability}
Let $\Pi$ be a finitely additive, exchangeable probability on 
$(\R^\infty, \mathcal B(\R)^{\otimes\infty})$ equipped with a full 
conditional probability $\Pi(\cdot \mid \cdot)$ in the sense of 
de~Finetti--R\'enyi. Suppose that for every $n \ge 1$ and every 
$k \in \{0, \ldots, n\}$,
\[
\Pi\!\left(Y_{n+1} \in I_k(Y_{1:n}) \,\middle|\, Y_{1:n}\right) 
= \frac{1}{n+1} \quad \text{a.s.}
\]
Then $\Pi$ is not conglomerable: there exist an event $B$ and a 
countable measurable partition $\{A_m\}_{m \ge 1}$ with $\Pi(A_m) > 0$ 
for all $m$ such that
\[
\Pi(B) < \inf_{m \ge 1} \Pi(B \mid A_m).
\]
Consequently, by Theorem~\ref{thm:heath_sudderth}, $\Pi$ admits a 
Dutch book.
\end{theorem}

\begin{proof}
See Appendix~\ref{app:lane_sudderth} for the explicit construction of 
$B$ and $\{A_m\}$. The key step is that the uniform rank constraints, 
applied at all sample sizes simultaneously, create an incompatibility 
between unconditional and conditional assessments that cannot be 
resolved within finite additivity while preserving conglomerability.
\end{proof}

\subsubsection*{Transfer to kernelized conformal prediction}

The connection to CP is now immediate. In the no-covariate case with a 
monotone score, CP's rank pivot induces the same order-statistic 
partition (Proposition~\ref{prop:cp_is_hill}), and the 
$A_{(n)}$-bridge (Definition~\ref{def:cp_bridge_kernel}) enforces 
exactly \eqref{eq:An_cell_assignment}.

\begin{corollary}[Nonconglomerability of bridged conformal prediction]
\label{cor:cp_nonconglom}
Any finitely additive exchangeable law on $\R^\infty$ whose one-step 
predictive assessments agree with the $A_{(n)}$-bridged conformal 
mechanism for all $n$ fails conglomerability and admits a Dutch book.
\end{corollary}

\begin{remark}[What the obstruction does and does not concern]
The obstruction is not marginal coverage, CP delivers that by 
construction. The obstruction concerns what happens when one promotes 
CP outputs to a sequential belief state: the resulting object cannot 
be conglomerable if it preserves the rank-calibration constraints at 
all stages. This is the precise sense in which "updating the CP belief" 
is problematic as an interpretation of sequential conditioning.
\end{remark}

\subsection{Consequences: no coherent \texorpdfstring{$\sigma$}{σ}-additive joint law}
\label{subsec:cp_no_joint}

Nonconglomerability is a finite-additivity pathology. A stronger statement, and the one
we exploit later, is that the rank-gap constraints \eqref{eq:An_cell_assignment} cannot
arise as the regular conditional distributions of any \emph{countably additive}
exchangeable law on $\R^\infty$.

\begin{lemma}[Countable additivity implies conglomerability]
\label{lem:sigma_additive_conglomerable}
Let $P$ be a countably additive probability on a standard Borel space, and let
$\{A_m\}_{m\ge1}$ be a countable measurable partition with $P(A_m)>0$ for all $m$.
Let $P(\cdot\mid A_m)$ be the usual conditional probabilities. Then for every event $B$,
\[
P(B)\ \in\ \big[\inf_{m\ge1} P(B\mid A_m),\ \sup_{m\ge1} P(B\mid A_m)\big].
\]
In particular, $P$ cannot satisfy $P(B)<\inf_m P(B\mid A_m)$ or $P(B)>\sup_m P(B\mid A_m)$.
\end{lemma}

\begin{proof}
By the law of total probability,
\[
P(B)=\sum_{m\ge1} P(B\cap A_m)=\sum_{m\ge1} P(B\mid A_m)\,P(A_m),
\]
which is a convex combination of the numbers $\{P(B\mid A_m)\}_{m\ge1}$ with weights
$\{P(A_m)\}_{m\ge1}$. Any convex combination lies in the closed interval between the
infimum and supremum of the support values.
\end{proof}

\begin{theorem}[No $\sigma$-additive extension of $A_{(n)}$ / bridged conformal ranks]
\label{thm:no_joint}
There does not exist a countably additive exchangeable probability measure
$P$ on $(\R^\infty,\mathcal B(\R)^{\otimes\infty})$ such that, for every $n\ge1$ and every
$k\in\{0,\ldots,n\}$,
\[
P\!\left(Y_{n+1}\in I_k(Y_{1:n}) \,\middle|\, Y_{1:n}\right)
\;=\;\frac{1}{n+1}
\quad\text{a.s.}
\]
Equivalently, there is no $\sigma$-additive exchangeable joint law whose one-step
regular conditionals reproduce the no-covariate $A_{(n)}$ rank calibration at all
sample sizes.
\end{theorem}

\begin{proof}[Proof by contradiction via conglomerability]
Assume for contradiction that such a countably additive exchangeable $P$ exists.
Since $(\R^\infty,\mathcal B(\R)^{\otimes\infty})$ is a standard Borel space, $P$ admits
regular conditional probabilities, and these conditionals satisfy the stated equal-gap
constraints by hypothesis.

Now apply Lane--Sudderth's theorem (Theorem~\ref{thm:conglomerability}) to the sequential
predictive system induced by these regular conditionals. Their result yields an event
$B$ and a countable measurable partition $\{A_m\}_{m\ge1}$ with $P(A_m)>0$ such that
\[
P(B)\;<\;\inf_{m\ge1} P(B\mid A_m).
\]
But this is impossible under countable additivity: by Lemma~\ref{lem:sigma_additive_conglomerable},
a countably additive $P$ must satisfy
$P(B)\in[\inf_m P(B\mid A_m),\sup_m P(B\mid A_m)]$. This contradiction shows no such
$\sigma$-additive exchangeable $P$ can exist.
\end{proof}

\begin{remark}[What the theorem does and does not say]
At each fixed $n$, the rank-gap assignment \eqref{eq:An_cell_assignment} is a perfectly
legitimate conditional probability assignment on the finite algebra generated by the
random order-statistic partition. The obstruction in Theorem~\ref{thm:no_joint} is
\emph{sequential}: one cannot choose such assignments for all $n$ so that they arise as
regular conditional distributions under a single countably additive exchangeable joint
law on the infinite product space.
\end{remark}

\subsubsection*{Pragmatic Interpretation}
Together, Theorems~\ref{thm:conglomerability} and \ref{thm:no_joint} formalize the sense
in which conformal prediction aligns with the fiducial/direct-probability tradition:
it is a rank-calibrated predictive assessment with finite-sample guarantees, but when
\emph{promoted} (via a bridge) to a sequential predictive object it encounters the same
nonconglomerability and non-extendability obstructions identified by Lane and Sudderth
for $A_{(n)}$-type rules. This is precisely why CP is reliable as a calibrated wrapper
at a fixed stage, but cannot generally serve as a coherent probabilistic state variable
for sequential decision-making without imposing additional structure that breaks the
distribution-free premise.

\section{Impossibility Results}
\label{sec:impossibility}

This section collects formal separations between conformal prediction (CP) and
Bayesian predictive inference, stated in the measure-theoretic language of
Section~\ref{sec:foundations}. Bayesian prediction produces a regular
conditional probability kernel (hence $\sigma$-additive across measurable
events), while CP is natively a rank-calibrated \emph{set-valued} acceptance
rule and only becomes a full predictive distribution after an extra
\emph{kernelization} choice (e.g.\ conformal predictive distributions via
randomized ranks).

Two tables structure the operational reading of the theorems.
Table~\ref{tab:so_what} maps each separation to a concrete engineering failure
mode and a safe-use mitigation. Table~\ref{tab:dependency_map} records how the
separations reinforce each other logically, so that one does not treat them as
independent objections. 

\subsection{Non-extensionality}
\label{subsec:nonextensionality}

Extensionality (Definition~\ref{def:extensionality}) requires that the induced
predictive rule depend on the data only through the conditional law
$P(Y\mid X)$, not through the design distribution $P(X)$. Split conformal
prediction violates this principle whenever the conditional law is
heteroscedastic and the calibration step pools residual magnitudes \emph{over
the marginal $P(X)$}.

\begin{theorem}[Non-extensionality of split conformal prediction]
\label{thm:extensionality}
There exist two joint laws $P_1$ and $P_2$ on $\mathcal X\times\mathcal Y$ with
the same conditional distribution $P_1(Y\in\cdot\mid X)=P_2(Y\in\cdot\mid X)$
but different marginals $P_1(X)\neq P_2(X)$ such that the split conformal
prediction sets differ with positive probability (as random set-valued maps of
$(x;Z_{1:n})$), even at the same test covariate $x$.
\end{theorem}

\begin{proof}
Let $\mathcal X=[0,1]$ and $\mathcal Y=\mathbb R$. Fix a nonconstant measurable
scale function $\sigma(x)=1+x$. Define a common conditional kernel
\[
K(\,\cdot\,\mid x)\;=\;N(0,\sigma(x)^2).
\]
Define two joint laws $P_1,P_2$ by choosing different covariate marginals and
the same conditional kernel:
\[
P_1:\ X\sim \mathrm{Unif}[0,1],\qquad
P_2:\ X\sim \mathrm{Beta}(2,5),
\qquad
Y\mid X=x \sim K(\cdot\mid x).
\]
Then $P_1(Y\mid X)=P_2(Y\mid X)$ by construction, but $P_1(X)\neq P_2(X)$.

Run split conformal with a fixed split $\mathcal I_1,\mathcal I_2$ with
$|\mathcal I_2|=m$, and take the (correct) regression fit $\widehat f\equiv 0$
(fitted or fixed). With absolute residual scores,
\[
R_i \;=\; |Y_i-\widehat f(X_i)|\;=\;|Y_i|,\qquad i\in\mathcal I_2.
\]
Conditional on $X_i=x$, $R_i\stackrel{d}{=}\sigma(x)\,|Z|$ with $Z\sim N(0,1)$.
Hence the unconditional residual distribution under $P_j$ is the scale-mixture
\[
\mathrm{Law}_{P_j}(R)\;=\;\int \mathrm{Law}\bigl(\sigma(x)\,|Z|\bigr)\,P_j(dx),
\qquad j\in\{1,2\}.
\]
Because $\sigma$ is nonconstant and $P_1(X) \neq P_2(X)$, the induced residual laws 
$\mathrm{Law}_{P_1}(R)$ and $\mathrm{Law}_{P_2}(R)$ are distinct probability measures on 
$\mathbb{R}_+$. To see this, note that for any $r > 0$,
\[
F_j(r) := P_j(\sigma(X)|Z| \leq r) = \int_0^1 \Phi\bigl(r/\sigma(x)\bigr) \, P_j(dx),
\]
where $\Phi$ is the standard normal CDF. Since $x \mapsto \Phi(r/\sigma(x))$ is strictly 
monotone in $x$ for each $r > 0$ and $P_1 \neq P_2$, we have $F_1 \neq F_2$ as functions. 
Distinct continuous CDFs on $\mathbb{R}_+$ cannot agree at all quantile levels: if 
$F_1(\xi) = F_2(\xi) = 1 - \alpha$ for all $\alpha \in (0,1)$, then $F_1 = F_2$. 
Hence there exists $\alpha \in (0,1)$ such that $q_{1-\alpha}^{(1)} \neq q_{1-\alpha}^{(2)}$.

Therefore there exists an $\alpha\in(0,1)$ such that the population
$(1-\alpha)$-quantiles differ: $q^{(1)}_{1-\alpha}\neq q^{(2)}_{1-\alpha}$.

Let $\widehat q$ be the empirical $(1-\alpha)$ quantile of $(R_i)_{i\in\mathcal I_2}$
(with any fixed tie convention). By Glivenko--Cantelli for the empirical CDF on
$\mathbb R_+$, $\widehat q\to q^{(j)}_{1-\alpha}$ in probability under $P_j$.
Therefore, for $m$ large enough, there is a threshold $q^\star$ strictly
between $q^{(1)}_{1-\alpha}$ and $q^{(2)}_{1-\alpha}$ such that
\[
\Prob_{P_1}\!\bigl(\widehat q < q^\star\bigr)\;>\;0
\quad\text{and}\quad
\Prob_{P_2}\!\bigl(\widehat q > q^\star\bigr)\;>\;0 .
\]
But for any fixed test covariate $x_0\in[0,1]$, split conformal returns the interval
\[
C_n(x_0;Z_{1:n}) = [-\widehat q,\widehat q].
\]
Hence the induced set-valued predictors differ with positive probability under
$P_1$ versus $P_2$ while $P(Y\mid X)$ is identical, violating extensionality. For extended derivations and additional discussion see ~\ref{thm:extensionality_app}.
\end{proof}

\begin{remark}[Operational reading]
The mechanism is exactly the one summarized in Table~\ref{tab:so_what}:
calibration is performed against a marginal residual law, which is a $P(X)$-mixture
under heteroscedasticity. Thus CP can change across deployments under covariate shift
even when the conditional kernel $P(Y\mid X)$ is stable. This is the concrete
``transport risk'' encoded by non-extensionality.
\end{remark}

\subsection{A Baire-category separation for kernelized CP}
\label{subsec:baire}

Theorems~\ref{thm:conglomerability}--\ref{thm:no_joint} show that demanding a
single $\sigma$-additive joint law with conformal/Hill ranks as its regular
conditionals is inconsistent. The right object for ``CP vs.\ Bayes'' at this
level is therefore the \emph{family of one-step predictive kernels}. The
separation result below is intentionally \emph{structural}: it does not rely on
a special ``with vs.\ without covariates'' story, but on a generic rigidity
property shared by conformal mechanisms, namely, that the history is compressed
through a \emph{calibration $\sigma$-field} generated by score comparisons (and
possibly additional randomization or side-information).

\subsubsection*{Kernel space as a Baire topology)}
Let $\mathcal Y$ be Polish and let $\mathcal P(\mathcal Y)$ carry the weak
topology. Fix once and for all:
(i) a countable set of bounded Lipschitz test functions
$\{\varphi_j\}_{j\ge1}\subset \mathrm{BL}_1(\mathcal Y)$ that is dense in
$\mathrm{BL}_1(\mathcal Y)$ under $\|\cdot\|_{\mathrm{BL}}$, and
(ii) for each $n\ge1$ a non-atomic Borel probability $\lambda_n$ on $\mathcal Y^n$
with full support.

For each $n$, let $\mathcal K_n$ be the set of measurable kernels
$\kappa_n:\mathcal Y^n\to\mathcal P(\mathcal Y)$. Associate to $\kappa_n$ the
sequence of bounded measurable functions
\[
T_n(\kappa_n)\;\defeq\;\bigl(f_{n,j}\bigr)_{j\ge1},
\qquad
f_{n,j}(y_{1:n}) \defeq \int_{\mathcal Y}\varphi_j(y)\,\kappa_n(dy\mid y_{1:n}).
\]
Define the metric
\begin{equation}
\label{eq:kernel_metric_L2}
d_n(\kappa_n,\kappa_n')\;\defeq\;
\sum_{j\ge1}2^{-j}\Bigl(\,
\|f_{n,j}-f'_{n,j}\|_{L^2(\lambda_n)}\wedge 1
\Bigr),
\end{equation}
where $f'_{n,j}$ corresponds to $\kappa_n'$.
Let $\mathcal K\defeq\prod_{n\ge1}\mathcal K_n$ with the product metric
\[
d(\kappa,\kappa')\defeq\sum_{n\ge1}2^{-n}\bigl(d_n(\kappa_n,\kappa_n')\wedge 1\bigr).
\]
(With this choice, $\mathcal K$ is a Baire space; for instance it can be realized
as a $G_\delta$ subset of a complete metric product of $L^2$-coordinate embeddings.)

\begin{definition}[Bayesian and conformal kernel classes]
\label{def:kernel_classes}
Let $\mathcal{B} \subset \mathcal{K}$ denote the set of \emph{Bayesian predictive kernels}: 
$\kappa \in \mathcal{B}$ if there exists a prior $Q$ on $\mathcal{P}(\mathcal{Y})$ such that, 
for all $n$ and all histories $y_{1:n}$,
\[
\kappa_n(B \mid y_{1:n}) = \int_{\mathcal{P}(\mathcal{Y})} \mu(B) \, Q(d\mu \mid y_{1:n}),
\quad B \in \mathcal{B}(\mathcal{Y}),
\]
where $Q(\cdot \mid y_{1:n})$ is the posterior on $\mathcal{P}(\mathcal{Y})$ obtained by 
updating $Q$ under the i.i.d.\ sampling model $Y_1, \ldots, Y_n \mid \mu \stackrel{\text{iid}}{\sim} \mu$.

Fix a \emph{single} conformal mechanism: a score construction (possibly depending on 
covariates and auxiliary randomness) together with a kernelization convention 
(e.g., randomized ranks or conformal predictive distributions). Let 
$\mathcal{C} \subset \mathcal{K}$ denote the corresponding class of (kernelized) 
conformal predictive kernels induced by this mechanism.

\begin{remark}[What determines $\mathcal{C}$]
\label{rem:what_determines_C}
The class $\mathcal{C}$ depends on the choice of score function, any data-splitting 
or fitting procedure, the tie-breaking convention, and the kernelization rule that 
promotes set-valued outputs to probability measures. Different choices yield different 
classes $\mathcal{C}$; Theorem~\ref{thm:meagre} applies to each such class individually, 
and Remark~\ref{rem:countable_mechanisms} extends the result to countable families.
\end{remark}
\end{definition}

\subsubsection*{Calibration $\sigma$-fields and conformal rigidity}
The defining structural feature of $\mathcal C$ is that the history is not used
in its full Borel complexity; instead it enters through the $\sigma$-field
generated by the calibration summary produced by the conformal mechanism.

\begin{definition}[Calibration $\sigma$-field at stage $n$]
\label{def:calib_sigma}
Fix $n$ and a conformal mechanism. Let $\mathcal G_n$ be the $\sigma$-field on
$\mathcal Y^n$ generated by the \emph{calibration summary} at stage $n$, that is,
by whatever objects the mechanism uses to compute the prediction set/distribution
(e.g.\ score comparisons or ranks, bin or stratum membership, local weights,
tie-break variables, fitted-model state, or other measurable summaries). We call
$\mathcal G_n$ the \emph{calibration $\sigma$-field}.
\end{definition}

\begin{definition}[$\mathcal G_n$-measurable kernels at stage $n$]
\label{def:G_measurable}
A kernel $\kappa_n\in\mathcal K_n$ is \emph{$\mathcal G_n$-measurable} if for
every $j\ge1$ the map
\[
y_{1:n}\ \longmapsto\ \int \varphi_j(y)\,\kappa_n(dy\mid y_{1:n})
\]
is $\mathcal G_n$-measurable.
\end{definition}

For a fixed mechanism, kernelized conformal predictive distributions are
$\mathcal G_n$-measurable by construction: once the calibration summary (and any
tie-breaking randomness) is fixed, the output kernel is determined.

\begin{theorem}[Meagre intersection of conformal and Bayesian kernels]
\label{thm:meagre}
Assume that for some $n$ the calibration $\sigma$-field $\mathcal G_n$ is a
proper sub-$\sigma$-field of $\mathcal B(\mathcal Y^n)$ (equivalently, the
conformal mechanism discards some Borel information in the history at stage $n$).
Then in the kernel space $(\mathcal K,d)$, the intersection $\mathcal B\cap\mathcal C$
is meagre.
\end{theorem}

\begin{proof}
Fix $n$ as in the hypothesis. Let $L^2(\lambda_n;\mathcal G_n)$ denote the closed
subspace of $\mathcal G_n$-measurable functions in $L^2(\lambda_n)$.
Define
\[
\mathsf{GM}_n \;\defeq\; \Bigl\{\kappa_n\in\mathcal K_n:\ f_{n,j}\in
L^2(\lambda_n;\mathcal G_n)\ \text{for all }j\ge1\Bigr\},
\]
where $f_{n,j}$ is the $j$th coordinate of $T_n(\kappa_n)$.

Because $L^2(\lambda_n;\mathcal G_n)$ is closed, $\mathsf{GM}_n$ is closed in
$(\mathcal K_n,d_n)$. If $\mathcal G_n\subsetneq\mathcal B(\mathcal Y^n)$, then
$L^2(\lambda_n;\mathcal G_n)$ is a proper closed linear subspace of $L^2(\lambda_n)$
and has empty interior. Consequently $\mathsf{GM}_n$ has empty interior in $\mathcal K_n$;
since it is closed, it is nowhere dense.

For the fixed conformal mechanism, $\mathcal C_n\subseteq \mathsf{GM}_n$ at stage $n$,
where $\mathcal C_n$ is the stage-$n$ projection of $\mathcal C$.
Let $\pi_n:\mathcal K\to\mathcal K_n$ be the continuous coordinate projection and
define $\mathcal D_n\defeq\pi_n^{-1}(\mathsf{GM}_n)$. Then $\mathcal D_n$ is nowhere
dense in $\mathcal K$, and $\mathcal C\subseteq \mathcal D_n$. Therefore
\[
\mathcal B\cap\mathcal C\ \subseteq\ \mathcal D_n
\]
is contained in a nowhere dense set and is meagre. See ~\ref{app:baire} for extended derivations and discussion.
\end{proof}

\begin{remark}[Recovering the usual ``rank'' story as a corollary]
\label{rem:rank_story}
In the no-covariate case, $\mathcal G_n$ may be taken as the rank $\sigma$-field
generated by score order relations (and tie-breaking variables). In split conformal
with covariates, $\mathcal G_n$ is generated by the fitted model, calibration residuals,
and residual order statistics (possibly within strata/bins or under local weights).
In both cases $\mathcal G_n\subsetneq\mathcal B(\mathcal Y^n)$ typically holds, so
Theorem~\ref{thm:meagre} applies: the coincidence ``kernelized CP = Bayes'' is topologically
exceptional for structural reasons, not because of a particular covariate/no-covariate
artifact.
\end{remark}

\begin{remark}[Countable families of mechanisms]
\label{rem:countable_mechanisms}
If one considers a \emph{countable} menu of conformal mechanisms (e.g.\ a countable
parametric family of scores/kernelizations), the union of the corresponding meagre
intersections $\bigcup_{r\ge1}(\mathcal B\cap\mathcal C^{(r)})$ remains meagre.
Thus, no countable ensemble of conformal procedures can generically recover Bayesian semantics.
\end{remark}

\subsection{Deficiency and Blackwell comparison}
\label{subsec:deficiency}

We now formalize the information-theoretic sense in which rank reductions discard
decision-relevant information in dominated regular experiments. This is the
experiment-theoretic underpinning of the ``loss-optimality gap'' row in
Table~\ref{tab:so_what}.

Let $\mathcal E=\{P_\theta:\theta\in\Theta\}$ be a dominated statistical experiment
on $(\mathcal Y^n,\mathcal B(\mathcal Y)^{\otimes n})$. Let $T:\mathcal Y^n\to\mathcal T$
be a statistic and $\mathcal E_T$ the reduced experiment $\{P_\theta\circ T^{-1}\}$.

\begin{theorem}[Positive deficiency under nonsufficiency]
\label{thm:lecam}
If $T$ is not sufficient in the Blackwell sense (equivalently, $\mathcal E_T$
is not Blackwell-equivalent to $\mathcal E$), then the reduced experiment has
strictly positive deficiency relative to the full experiment:
\[
\delta(\mathcal E_T,\mathcal E)\;>\;0.
\]
In particular, if $T$ is the rank $\sigma$-field (or any statistic measurable
with respect to it) underlying a rank-calibrated conformal reduction in a
regular dominated model, then there exist bounded-loss decision problems for
which every $T$-measurable procedure is uniformly dominated by a procedure
based on the full data, and hence by a Bayesian rule in a complete class.
\end{theorem}

\begin{proof}[Proof sketch]
Blackwell's comparison of experiments implies that $\mathcal E_T$ is
Blackwell-inferior to $\mathcal E$ whenever $T$ is not sufficient; equivalently,
there exists a bounded-loss decision problem for which the optimal risk under
$\mathcal E_T$ is strictly larger than the optimal risk under $\mathcal E$
\citep{blackwell1951comparison}. Le Cam's deficiency admits a variational
characterization as the maximal (minimal) risk gap over bounded decision problems
\citep[see, e.g.,][]{lecam1986asymptotic,lecam2000asymptotics}. Therefore strict
Blackwell inferiority implies $\delta(\mathcal E_T,\mathcal E)>0$. Extended discussion is presented in~\ref{app:deficiency}.
\end{proof}

\begin{corollary}[Rank-based predictors fall outside complete classes in regular models]
\label{cor:inadmissible}
In any dominated regular model where the conformal rank statistic (or any
rank-generated reduction) is not sufficient, purely rank-based predictive
procedures are not in a complete class: there exist Bayesian procedures with
strictly smaller risk for some bounded loss problems.
\end{corollary}

\subsection{Practice implications}
\label{subsec:so_what_table}

The separation results are not merely foundational curiosities, they have direct operational consequences. Modern machine learning pipelines increasingly treat uncertainty estimates as first-class objects: prediction sets feed into downstream decisions, get cached and reused across deployments, and serve as inputs to multi-stage optimization. The temptation is to treat a conformal output as a coherent probabilistic state that can be composed across stages, transported across covariate distributions, and optimized under arbitrary objectives. Our results delimit precisely where this temptation leads to trouble.

The failure modes cluster into three categories. \emph{Transport failures} arise when CP outputs change under covariate shift even though the conditional $P(Y \mid X)$ remains stable (Theorem~\ref{thm:extensionality}); practitioners who deploy a calibrated model in a new environment may find that coverage degrades not because the underlying signal changed, but because the calibration was tied to design features that did. \emph{Sequential failures} arise when CP outputs are updated or conditioned as if they were beliefs: the rank-calibration constraints that guarantee marginal coverage also force nonconglomerability (Theorem~\ref{thm:conglomerability}), meaning that conditional assessments can collectively contradict the unconditional assessment, a structure that admits Dutch books and cannot arise from any $\sigma$-additive joint law (Theorem~\ref{thm:no_joint}). \emph{Decision-theoretic failures} arise when rank- or proxy-based reductions discard information that matters for downstream loss: unless the reduction happens to be sufficient, Blackwell/Le Cam theory guarantees a strict optimality gap for some bounded-loss problems (Theorem~\ref{thm:lecam}).

Table~\ref{tab:so_what} maps each separation theorem to its corresponding failure mode and a conservative mitigation strategy. The mitigations share a common structure: use CP for what it guarantees (finite-sample marginal coverage at a fixed stage) and carry belief-state functionality separately in a $\sigma$-additive model whose regular conditionals support composition and transport.

\begin{table}[t]
\centering
\footnotesize
\setlength{\tabcolsep}{4pt}
\renewcommand{\arraystretch}{1.05}
\caption{Practice map: separation results, failure modes, and mitigations.}
\label{tab:so_what}
\begin{tabularx}{\textwidth}{@{}>{\raggedright\arraybackslash}p{2.8cm}>{\raggedright\arraybackslash}X>{\raggedright\arraybackslash}X@{}}
\toprule
\textbf{Separation} & \textbf{Failure mode} & \textbf{Mitigation} \\
\midrule
Non-extensionality\newline (Thm.~\ref{thm:extensionality})
& Covariate shift: prediction sets change across deployments even when $P(Y \mid X)$ is stable, because calibration depends on $P(X)$.
& Stratify or use group-conditional conformal; treat CP as a wrapper around a model that carries the conditional semantics. \\[4pt]

Nonconglomerability\newline (Thm.~\ref{thm:conglomerability})
& Sequential incoherence: conditioning on later information can contradict earlier assessments; Dutch-book vulnerabilities under rich partitions.
& Do not use CP outputs as a dynamic belief state; for sequential decisions, maintain a $\sigma$-additive filter and apply conformal calibration as a terminal check. \\[4pt]

No $\sigma$-additive joint\newline (Thm.~\ref{thm:no_joint})
& No globally consistent uncertainty ledger: ``updating the CP belief'' is not well-defined as regular conditioning.
& Avoid posterior/belief language for CP outputs; use generative models for utility optimization and treat CP as a coverage constraint. \\[4pt]

Meagre CP--Bayes overlap\newline (Thm.~\ref{thm:meagre})
& Structural rarity: ``CP $\approx$ Bayes generically'' is false; coincidence requires special structure.
& Position CP as coverage-by-construction, not approximate Bayesian inference; hybrid methods must state and validate the structural assumptions under which coincidence holds. \\[4pt]

Positive deficiency\newline (Thm.~\ref{thm:lecam})
& Decision gap: rank reductions discard metric information; suboptimal under smooth losses unless ranks are sufficient.
& Start from a coherent predictive distribution; add CP as a diagnostic or constraint rather than replacing the predictive law. \\
\bottomrule
\end{tabularx}
\end{table}

Table~\ref{tab:dependency_map} records how the separations reinforce one another. The logical structure is worth emphasizing: these are not five independent objections but manifestations of a single architectural mismatch. Nonconglomerability is the operational symptom of non-extendability to a $\sigma$-additive joint; non-extendability implies that Bayesian coincidence is structurally exceptional; and rank-based coarsening, which underlies both CP calibration and the sequential obstructions, is precisely what drives positive deficiency. Non-extensionality stands somewhat apart, it concerns dependence on design rather than sequential coherence, but it shares the same root cause: calibration against a marginal distribution that is not determined by the conditional kernel alone.

\begin{table}[t]
\centering
\small
\setlength{\tabcolsep}{6pt}
\renewcommand{\arraystretch}{1.15}
\caption{Logical linkage: the separations are not independent objections but consequences of a single structural mismatch.}
\label{tab:dependency_map}
\begin{tabularx}{\textwidth}{@{}>{\raggedright\arraybackslash}p{4.2cm}>{\raggedright\arraybackslash}X@{}}
\toprule
\textbf{Implication} & \textbf{Why it holds} \\
\midrule
Rank calibration\newline $\Rightarrow$ nonconglomerability
& The $A_{(n)}$ rank-gap constraints (Definition~\ref{def:cp_bridge_kernel}) force uniform conditional ranks at every stage. Lane--Sudderth show this is incompatible with conglomerability under any finitely additive extension. \\[4pt]

Nonconglomerability\newline $\Rightarrow$ no $\sigma$-additive joint
& Countable additivity implies conglomerability (Lemma~\ref{lem:sigma_additive_conglomerable}). Contrapositive: nonconglomerable rank constraints cannot arise as regular conditionals of any $\sigma$-additive exchangeable law. \\[4pt]

No $\sigma$-additive joint\newline $\Rightarrow$ Bayes overlap is meagre
& Bayesian predictive kernels \emph{are} regular conditionals of $\sigma$-additive laws. If CP kernels cannot be, coincidence requires the CP mechanism to accidentally reproduce Bayesian structure, a condition satisfied only on a Baire-meagre set. \\[4pt]

Rank reduction\newline $\Rightarrow$ positive deficiency
& Calibration compresses data through a coarse $\sigma$-field (ranks, residual order statistics). Unless that $\sigma$-field is sufficient, rare in regular models, Blackwell's theorem guarantees strict information loss for some bounded-loss problems. \\[4pt]

Non-extensionality\newline $\Rightarrow$ transport fragility
& Calibration pools residuals over $P(X)$. Under heteroskedasticity, changing $P(X)$ changes the residual mixture and hence the calibration quantile, even when $P(Y \mid X)$ is fixed. \\
\bottomrule
\end{tabularx}
\end{table}

Together with the sequential obstructions established in Sections~\ref{subsec:cp_conglomerability}--\ref{subsec:cp_no_joint}, these results complete the separation picture. Conformal prediction delivers distribution-free marginal coverage, a genuine and valuable guarantee, but the induced objects do not behave like $\sigma$-additive posterior predictives that can be composed, transported, and optimized as a unified probabilistic state. The next section diagnoses where Prediction-Powered Inference inherits analogous tensions once it is positioned as ``black-box prediction $\to$ inference.''

\section{Case Study: Prediction-Powered Inference (PPI)}
\label{sec:ppi}

Prediction-Powered Inference (PPI) \citep{angelopoulos2023prediction} converts a
strong pre-trained predictor (often trained using abundant unlabeled features or
external supervision) plus a comparatively small labeled sample into valid
frequentist confidence intervals for population targets.  At a high level, PPI
treats a black-box prediction $\widehat m(X)$ as a high-accuracy proxy for $Y$,
then uses the labeled sample to \emph{correct} the proxy in a way that preserves
nominal coverage under appropriate conditions.

Our aim is not to dispute PPI as an engineering tool.  Rather, we position PPI
within the same foundational map developed for conformal prediction in
Section~\ref{sec:cp} and the separations in Section~\ref{sec:impossibility}.  The
key distinction is ontological: PPI is natively a \emph{targeted confidence
procedure} (for a specified estimand and sampling plan), not a $\sigma$-additive
predictive kernel or ``implicit posterior.''  When PPI outputs are promoted to
a transportable probabilistic state, to be composed across stages, conditioned
on refinements of the design, or interpreted as an approximation to Bayesian
prediction, one encounters the same structural tensions: dependence on design
and training regimes (non-extensionality), instability under post hoc
conditioning, and experiment-theoretic gaps relative to coherent full-data
procedures.

\subsubsection*{PPI in the broader ``prediction $\to$ inference'' wave}
PPI is part of a fast-moving family of methods that attempt to convert
high-performing predictors (increasingly including foundation models) into
downstream uncertainty quantification and inference without committing to a
fully specified probabilistic data-generating model.  Such procedures can be
valuable as calibrated wrappers, but a wrapper is not a predictive kernel: in
the absence of additional modeling structure, the output should be treated as a
procedure with a stated frequentist guarantee, not as a belief state.

\subsection{PPI as algorithmic direct probability: the proxy is not extensional}
\label{subsec:ppi_fiducial}

A common PPI starting point is a regression-type decomposition
\begin{equation}
\label{eq:ppi_decomp}
Y \;=\; \widehat m(X) \;+\; R,
\end{equation}
where $\widehat m$ is a learned predictor and $R$ is the residual.  The
foundational point is that $\widehat m$ is typically a \emph{pipeline object}:
a functional of upstream training data, training design, algorithmic choices
(architecture, regularization, early stopping, pretraining distribution,
fine-tuning regime), and sometimes unlabeled deployment covariates.  Even when a
population regression function $m(x)=\E[Y\mid X=x]$ exists, the map
$P \mapsto \widehat m$ is generally \emph{not} determined by the conditional
kernel $P(Y\mid X)$ alone.

Measure-theoretically, $\widehat m$ is random and measurable with respect to a
$\sigma$-field generated by auxiliary training randomness/data, possibly coupled
to the unlabeled design sample.  Consequently, the induced residual law
\[
\mathrm{Law}(R\mid X) \;=\; \mathrm{Law}\bigl(Y-\widehat m(X)\mid X\bigr)
\]
is not, in general, a functional of $P(Y\mid X)$ alone: it depends on the joint
mechanism that generated $\widehat m$, and often also on the marginal design law
$P(X)$ through training/calibration.

\subsubsection*{Connection to the CP story}
In CP, non-extensionality arises because calibration pools residuals over the
marginal $P(X)$ under heteroscedasticity (Theorem~\ref{thm:extensionality}).
In PPI, the dependence is typically stronger: the learned proxy $\widehat m$ is
itself a non-likelihood ingredient, and its distribution depends on the
training and design regimes.  In this sense, PPI is a form of \emph{algorithmic
direct probability}: inference is produced by correcting an algorithmic proxy
rather than by conditioning within a single $\sigma$-additive joint model.

\subsection{Two-pipeline non-extensionality example: same \texorpdfstring{$P(Y \mid X)$}{P(Y|X)}, different PPI behavior}
\label{subsec:ppi_two_pipeline}

We record a minimal example making the non-extensionality concrete without
invoking covariate shift.

\begin{example}[Two pipelines, identical conditional law, different residual laws]
\label{ex:ppi_two_pipeline}
Let $\mathcal X=[0,1]$, $\mathcal Y=\mathbb R$, and suppose the target conditional
kernel is fixed:
\[
Y \;=\; f(X) + \varepsilon,\qquad \varepsilon\sim N(0,1)\ \text{independent of }X,
\]
so $P(Y\in\cdot\mid X=x)=N(f(x),1)$ is the same across deployments.

Consider two upstream learning pipelines that both output a proxy $\widehat m$,
but with different training regimes:
\[
\widehat m_A(x)=f(x)\quad\text{(well-specified proxy)},\qquad
\widehat m_B(x)=0\quad\text{(misspecified proxy)}.
\]
These two proxies can arise, for example, from different pretraining
distributions, model classes, or optimization choices, while the deployment
conditional law $P(Y\mid X)$ remains unchanged.

Then the residuals differ:
\[
R_A = Y-\widehat m_A(X)=\varepsilon,\qquad
R_B = Y-\widehat m_B(X)=f(X)+\varepsilon.
\]
Any PPI correction (and hence any resulting confidence interval width or
bias-correction magnitude) that depends on the residual distribution will
generally differ between pipelines A and B, despite identical $P(Y\mid X)$.
\end{example}

\noindent
Example~\ref{ex:ppi_two_pipeline} isolates the key transport risk: PPI behavior
can change across deployments because the proxy mechanism changes, not because
the conditional law changes.  This is the PPI analogue of
Theorem~\ref{thm:extensionality} for split conformal, but with a stronger
mechanism: the non-extensional ingredient is not merely pooled calibration, but
the learned decomposition itself.

\subsection{Conditioning and ``belief-state'' misuse: where incoherence enters}
\label{subsec:ppi_conditioning}

PPI is designed to provide valid coverage for a specified target under stated
conditions and a stated sampling plan.  However, if one informally interprets a
PPI output as encoding a \emph{full probabilistic belief state}, for example,
as if it defined an implicit posterior on the target that can be freely
conditioned on refinements of information about the design or subpopulations, then
one is implicitly asserting the existence of a coherent joint model under which
those refinements correspond to regular conditioning.

The structural issue parallels the CP story, but the mechanism differs.  In
many PPI constructions, the correction term aggregates residual information
globally over the design distribution of the labeled or unlabeled sample.  If a
user subsequently conditions on a refined design (e.g.\ a subpopulation or a
covariate-defined stratum), the relevant residual law changes, and the intended
correction need not transport unless one has explicitly modeled the joint
training--labeling--deployment mechanism.  Thus PPI is coherent as the targeted
confidence procedure it is designed to be, but it is not generally coherent as
a transportable probabilistic state variable across arbitrary filtrations.

\subsection{Le Cam and Blackwell comparison: proxy-based inference as a garbling}
\label{subsec:ppi_lecam}

From an experiment-theoretic perspective, PPI replaces (some of) the labeled
information in $Y$ by an algorithmic proxy $\widehat m(X)$ plus a correction
estimated from a labeled subsample.  Abstractly, ``working through a proxy'' is
a \emph{garbling} of the full-data experiment unless the proxy is sufficient.

In a dominated model for $(X,Y)$, observing $(X,\widehat m(X))$ (and a small
labeled correction sample) is generally Blackwell-inferior to observing $(X,Y)$
at the same effective labeled sample size.  Therefore, unless one imposes
conditions that make the proxy essentially sufficient for the target, there
exist bounded decision problems for which the proxy-based experiment has
strictly larger optimal risk, yielding a strictly positive deficiency gap as in
Theorem~\ref{thm:lecam}.  This is the PPI analogue of the ``positive deficiency''
row in Table~\ref{tab:so_what}.

\begin{theorem}[Deficiency gap for proxy-based reductions]
\label{thm:ppi_deficiency}
Consider a dominated experiment $\mathcal{E} = \{P_\theta : \theta \in \Theta\}$ for 
$(X, Y)$ and a proxy statistic $T = T(X, U)$ where $U$ represents auxiliary training 
randomness (e.g., $T = \widehat{m}(X)$ for a learned predictor $\widehat{m}$). If $T$ 
is not Blackwell-sufficient for $\theta$ in the full experiment, then the reduced 
experiment $\mathcal{E}_T = \{P_\theta \circ T^{-1}\}$ has strictly positive deficiency 
relative to $\mathcal{E}$:
\[
\delta(\mathcal{E}_T, \mathcal{E}) > 0.
\]
Consequently, there exist bounded-loss decision problems for which any decision rule 
depending only on $(T, \text{labeled correction})$ is uniformly dominated by a rule 
using full labeled data.
\end{theorem}

\begin{proof}
The proxy $T$ induces a Markov kernel (garbling) from the full observation $(X, Y)$ 
to the reduced observation $T(X, U)$. By Blackwell's comparison theorem, the reduced 
experiment $\mathcal{E}_T$ is equivalent to $\mathcal{E}$ if and only if this garbling 
admits a reverse, that is, if and only if $T$ is sufficient. When sufficiency fails, 
there exists a bounded-loss decision problem for which the minimax risk under 
$\mathcal{E}_T$ strictly exceeds that under $\mathcal{E}$; the deficiency 
$\delta(\mathcal{E}_T, \mathcal{E})$ is at least this risk gap. The argument is 
identical to Theorem~\ref{thm:lecam}; see Appendix~\ref{app:deficiency} for 
details.
\end{proof}

\begin{remark}[What this does \emph{not} say]
Theorem~\ref{thm:ppi_deficiency} does not negate PPI's stated coverage results for 
the targets and regimes where they hold. It says that \emph{across the space of 
decision problems}, proxy-based reductions are generically information-losing unless 
the proxy is sufficient, and hence one should not expect universal optimality or 
``implicit Bayes'' interpretations.
\end{remark}

\subsection{Practice map for PPI and links back to CP}
\label{subsec:ppi_tables}

Table~\ref{tab:ppi_so_what} mirrors Table~\ref{tab:so_what} but for PPI.  The
message is the same: one-shot frequentist validity for a specified target does
not automatically license interpreting the output as a coherent, transportable
belief state.

\begin{table}[t]
\centering
\footnotesize
\setlength{\tabcolsep}{4pt}
\renewcommand{\arraystretch}{1.15}
\caption{PPI practice map: structural tension $\to$ failure mode $\to$ safe use (parallel to Table~\ref{tab:so_what}).}
\label{tab:ppi_so_what}
\begin{tabularx}{\textwidth}{@{}l>{\raggedright\arraybackslash}X>{\raggedright\arraybackslash}X@{}}
\toprule
\textbf{Structural issue} & \textbf{Failure mode (where it bites)} & \textbf{Safe use / mitigation} \\
\midrule
Non-extensionality of the proxy $\widehat m$
& Inference changes across deployments under shifts in upstream training/design regimes even when the target $P(Y\mid X)$ is stable.
& Treat $\widehat m$ as part of the experiment; stress-test under design/training shift; stratify/correct within stable subpopulations. \\

Belief-state misuse under post hoc conditioning
& Conditioning on refined designs/subgroups can invalidate the intended correction; ``implicit posterior'' semantics become unstable.
& Use PPI for the stated target and sampling plan; avoid treating outputs as transportable posteriors across filtrations without an explicit joint model. \\

Random (proxy-dependent) criterion vs.\ classical $M$-estimation
& LAN/$M$-estimation heuristics can mislead if the experiment is not enlarged: the criterion is random through $\widehat m$ and depends on upstream training.
& State the expanded experiment explicitly; use influence-function/semiparametric analysis that conditions on (or models) the learned proxy mechanism. \\

Blackwell/Le Cam deficiency
& Proxy reductions discard information; for some loss problems, proxy-based procedures are uniformly dominated by full-data coherent rules.
& Use proxy-based inference when labels are scarce and the target is narrow; do not claim universal optimality or ``implicit Bayes.'' \\
\bottomrule
\end{tabularx}
\end{table}

\noindent
Section~\ref{sec:discussion} returns to hybrid pipelines and summarizes when
wrapper-style calibration is safe and when it is not.

\section{Discussion}
\label{sec:discussion}

This paper draws a sharp boundary between \emph{validity guarantees} and
\emph{belief semantics}. Conformal prediction (CP) delivers a bona fide guarantee:
finite-sample \emph{marginal} coverage under exchangeability. Prediction-powered
inference (PPI) delivers a complementary guarantee: valid confidence procedures
for specified targets that leverage a strong predictor together with a small
labeled sample, under stated stability conditions. These guarantees are genuine,
useful, and in many settings exactly what one should demand from a wrapper.
However, neither guarantee alone warrants interpreting the wrapper output
as a coherent \emph{belief state}, that is, as an object arising from a single
$\sigma$-additive joint law whose regular conditionals support sequential
updating, transport, and composition.

The conformal results in particular yield a taxonomy of distinct failure modes
that arise when calibrated sets (or rank-calibrated assessments) are used as if
they were belief objects. Non-extensionality identifies a transport failure
mode: even when $P(Y\mid X)$ is stable, the wrapper may change under shifts in
the covariate marginal, so the output is not determined by evidentially relevant
conditionals alone (Theorem~\ref{thm:extensionality}). Nonconglomerability and
non-extendability identify a sequential failure mode: one-step rank-calibrated
assessments do not generally embed into any $\sigma$-additive joint law with
regular conditionals reproducing those assessments across time, so ``updates''
based on wrapper objects need not be coherent (Theorems~\ref{thm:conglomerability}
--\ref{thm:no_joint}). The Baire-category separation identifies a structural
rarity mode: coincidence between calibrated-wrapper kernels and Bayesian
predictive semantics is topologically exceptional rather than generic
(Theorem~\ref{thm:meagre}). Finally, the Le~Cam/Blackwell comparison identifies a
decision-theoretic failure mode: rank- or proxy-based reductions are generically
not sufficient, hence incur positive deficiency for some bounded-loss decision
problems (Theorem~\ref{thm:lecam}). The PPI examples emphasize that these
boundaries are not CP-specific: whenever inference is routed through algorithmic
pivots, compressions, or proxies rather than through an explicitly modeled
$\sigma$-additive mechanism, extensionality and experiment comparison become the
natural diagnostics. In short, coverage can hold for the intended target while
global belief semantics, transportable updateability, and decision sufficiency
fail.

\subsection{Two families of ``impossibility'' results: what they rule out, and what remains}
\label{subsec:discussion_impossibility_map}

To connect our conclusions with the broader distribution-free (``random-world'')
predictive inference literature, it is helpful to separate two logically
distinct families of impossibility claims that are often discussed under a
single umbrella. The first family rules out strong truth-conditional ambitions
in highly assumption-light regimes; the second family (developed here) rules out
a different upgrade: manufacturing compositional semantics from validity alone.
Table~\ref{tab:impossibility_map} summarizes this division of labor.

\begin{table}[t]
\centering
\footnotesize
\setlength{\tabcolsep}{4pt}
\renewcommand{\arraystretch}{1.18}
\caption{Taxonomy of impossibility results in distribution-free predictive inference.
Impossibility I motivates why calibration is a natural \emph{guarantee object}.
Impossibility II (our contribution) clarifies why such guarantees do not generically
upgrade into compositional, $\sigma$-additive belief semantics.}
\label{tab:impossibility_map}
\begin{tabularx}{\textwidth}{@{}>{\raggedright\arraybackslash}p{0.21\textwidth}
                            >{\raggedright\arraybackslash}p{0.27\textwidth}
                            >{\raggedright\arraybackslash}p{0.26\textwidth}
                            >{\raggedright\arraybackslash}X@{}}
\toprule
\textbf{Strand} &
\textbf{Object ruled out (in general)} &
\textbf{Minimal escape hatch} &
\textbf{Implication} \\
\midrule

Impossibility I:\newline
truth-of-conditionals &
Uniform, assumption-free recovery of the \emph{true} conditional law (or a
universally valid surrogate for $P(Y\mid X)$) over rich, unconstrained worlds. &
Relax from truth-claims to \emph{validity constraints} (coverage/calibration), or
assume additional structure (smoothness, invariances, parametrics, stationarity). &
This supports the calibration-first pivot: validity is a sensible regulative target
when truth-conditional semantics is unattainable. \\

Impossibility I$'$:\newline
strong conditional validity &
Distribution-free \emph{exact conditional} guarantees (e.g., exact conditional
coverage for all $x$) with nontrivial efficiency. &
Settle for \emph{marginal} validity, group-conditional validity (Mondrian), or
introduce smoothing/model structure. &
This explains why marginal certificates are routinely the correct object; it also
motivates the question of how such certificates should (and should not) be used. \\

Impossibility II:\newline
semantics-from-guarantees &
Upgrading wrapper outputs into a $\sigma$-additive, compositional belief state that
supports transport, sequential conditioning, and decision sufficiency \emph{without}
enlarging the experiment. &
Either enlarge the experiment (model the training/design/deployment mechanism), or
treat wrappers as terminal guardrails rather than updateable belief objects. &
Formalized here via non-extensionality (Theorem~\ref{thm:extensionality}), 
nonconglomerability (Theorems~\ref{thm:conglomerability}--\ref{thm:no_joint}), 
positive deficiency (Theorem~\ref{thm:lecam}), and topological rarity of 
Bayesian coincidence (Theorem~\ref{thm:meagre}). \\
\bottomrule
\end{tabularx}
\end{table}

This framing clarifies how the impossibility results emphasized in the
game-theoretic and algorithmic random-world CP literature 
\citep{vovk2005algorithmic,shafervovk2008tutorial} relate to our separations. 
Those results motivate the move from strong semantic targets (e.g., 
truth-conditional or exact conditional claims) to weaker but auditable
validity certificates (Impossibility I / I$'$). Our contribution is orthogonal:
even after granting calibration as the right \emph{guarantee object}, we show
that one should not expect those guarantees to \emph{become} compositional belief
semantics or decision-sufficient reductions absent an explicit model of the
full mechanism generating the certificate (Impossibility II).

\subsection{Topological rarity as an assumption-light obstruction}
\label{subsec:discussion_topology}

The Baire-category results should be read as an assumption-light obstruction to
a common interpretive slide: treating ``calibrated wrapper'' as synonymous with
``Bayesian-like semantics.'' The result does not claim that CP--Bayes coincidence
is unlikely under any particular data-generating process; rather, it identifies
the coincidence regime as \emph{structurally exceptional} within natural spaces
of predictive kernels (Theorem~\ref{thm:meagre}). Consequently, the burden shifts
from rhetoric to verification: if a given application requires sequential
$\sigma$-additive semantics, then one must impose and validate structural
conditions (e.g., sufficiency or invertibility of the calibration $\sigma$-field)
that place the problem inside an exceptional regime where such semantics can in
fact emerge.

\subsection{Why deficiency matters beyond coverage plots}
\label{subsec:discussion_deficiency}

Coverage is a statement about one performance functional. Deficiency is a
statement about \emph{uniform} decision performance across bounded losses. When
a wrapper factors through a coarsening (ranks, residual summaries, proxy
compressions), experiment comparison predicts that there exist decision problems
for which no wrapper-measurable rule can match what is achievable from the
richer experiment (Theorem~\ref{thm:lecam}). This distinction is especially
salient in compositional pipelines (screening/triage, control, reinforcement
learning, multi-stage auditing), where small information losses can compound into
systematic regret. For this reason, empirical evaluation should not stop at
coverage-versus-width summaries: it should include stress tests that quantify
regret amplification under transport and sequential composition.

\subsection{Why post-hoc patches cannot generically repair semantic loss}
\label{subsec:discussion_patching}

A natural response to the foregoing is to treat CP (or CP-like artifacts) as
intermediate objects and apply downstream patches, smoothing, kernelization,
temperature scaling, ensembling, recalibration, or proxy-based debiasing, to
restore probabilistic meaning. The topological framing clarifies when such
repairs are formally blocked. If at time $n$ a pipeline stage outputs an object
measurable with respect to a proper calibration $\sigma$-field
$\mathcal{G}_n=\sigma(T_n(Z_{1:n}))$, then any post-processing that factors
through $T_n$ remains $\mathcal{G}_n$-measurable. In that case, the attainable
kernel class remains confined to the same Baire-meagre region of kernel space
(Theorem~\ref{thm:meagre} and its countable-family corollary): ex-post maps that
do not re-inject the lost information cannot generically manufacture
$\sigma$-additive belief semantics or sequential realizability. What such patches
\emph{can} do is improve performance within the wrapper-measurable class (e.g.,
width, stability, calibration under specific shifts), but not erase the
structural distinction between guarantees and semantics.

\subsection{Scope and governance}
\label{subsec:discussion_takeaway}

The practical message is not anti-calibration. It is a governance message about
\emph{proper scope}. Calibration and proxy corrections are best understood as
regulative constraints and terminal certificates for specified targets. When the
task demands transportable, compositional uncertainty representations that support
conditioning and sequential updating, one must supply the corresponding semantic
structure explicitly (typically via a $\sigma$-additive model of the joint
mechanism), and treat wrappers as safeguards rather than as substitutes for
belief.

A complementary semantics program makes the same point from the opposite
direction: rather than forcing calibrated wrappers to approximate a single
$\sigma$-additive joint law, one can treat conformal prediction natively as a
set-valued/credal correspondence and organize its composition principles on that
basis (e.g., categorically) \citep{caprio2025joys}. Either way, the governance
principle is the same: \emph{declare the semantic carrier}. If the downstream
workflow requires $\sigma$-additive conditioning and decision-theoretic
comparisons across experiments, then a $\sigma$-additive kernel must carry the
semantics and calibration remains a guardrail. If the workflow instead accepts
imprecise/credal semantics, then the set-valued object is primitive and one
should reason with the corresponding composition calculus.

A constructive research direction follows: characterize maximal verifiable
regimes in which compositional semantics emerges despite the topological fragility
identified here.

\subsection{Towards a Bayesian--CP interface}
\label{subsec:bayes_cp_interface}

\subsubsection*{Bayesian predictive kernels as the semantic carrier.}
The object that carries belief semantics is a $\sigma$-additive sequential
predictive kernel family $\{K_n\}_{n\ge 1}$, where each
$K_n(\cdot \mid y_{1:n})$ is a probability measure on $\mathcal Y$ and is
measurable in the history. When such a family arises as the regular conditional
distributions of a single $\sigma$-additive probability law on $\mathcal Y^\infty$,
it supports coherent composition: tower-consistent updating, conditioning under
refinements of the filtration, and decision-making under arbitrary downstream
objectives.

Under exchangeability, de~Finetti’s representation yields the familiar mixture
form
\[
K_n(B \mid y_{1:n}) = \int \mu(B)\, Q(d\mu \mid y_{1:n}),
\qquad B \in \mathcal B(\mathcal Y).
\]
Making the kernel explicit thus serves as an \emph{interface contract}: if an
uncertainty object is to be transported, conditioned, or optimized downstream,
it must be anchored in such a $\sigma$-additive state (or embedded into an
enlarged one).

\subsubsection*{What Bayesian--CP hybrids can and cannot promise.}
A conformal layer can legitimately serve as a calibration guardrail on top of a
semantic kernel. In this design, the Bayesian model supplies the compositional
belief semantics (via $\{K_n\}$), while conformalization is used as a terminal
certificate to enforce finite-sample validity properties for specified targets
(e.g., by conformalizing residuals, posterior predictive draws, or other
model-derived scores). This interface has been developed in several forms in
the recent literature; see, for example,
\citet{deliu2025interplay,fong2021conformal,deshpande2024online}.

Our separation results delimit what calibration alone cannot do. In particular,
a purely rank-calibrated wrapper cannot, without additional modeling structure,
\emph{generically} manufacture a $\sigma$-additive sequential belief state, nor
can it recover the full decision content of the unreduced experiment under
Blackwell/Le~Cam comparison. This is not a competitive claim but a structural
one: Bayesian kernels and conformal procedures are different objects, optimized
for different desiderata. Hybrids can be principled precisely when they keep the
division of labor explicit: kernels carry semantics; conformalization certifies
validity relative to that semantic carrier.
\subsubsection*{Empirical agenda}
The theory motivates empirical evaluations targeted at specific coherence gaps,
not solely marginal coverage:
\begin{enumerate}[label=(\roman*),nosep]
\item \emph{Transport}: hold $P(Y \mid X)$ fixed while shifting $P(X)$, and measure
output drift and downstream loss, diagnosing violations of conditional
extensionality (Theorem~\ref{thm:extensionality});
\item \emph{Sequential updateability}: embed CP-style objects in multi-stage
pipelines and test instability under conditioning refinements
(Theorems~\ref{thm:conglomerability}--\ref{thm:no_joint});
\item \emph{Decision loss}: compare rank- or proxy-based reductions to
full-kernel procedures under bounded losses to expose regret consistent with
Le~Cam and Blackwell deficiency gaps (Theorem~\ref{thm:lecam}).
\end{enumerate}
These diagnostics become essential once uncertainty objects are composed,
transported, and optimized across stages.

\section{Conclusion}

Conformal prediction is a calibration device in the fiducial/direct-probability lineage: it guarantees finite-sample marginal coverage under exchangeability by rank construction, but it does not, in general, supply the semantics of a $\sigma$-additive sequential belief state. Our separation results, non-extensionality, nonconglomerability, non-extendability, topological rarity of Bayesian coincidence, and positive deficiency under rank/proxy reductions, formalize this distinction and identify concrete failure modes when calibrated wrappers are treated as transportable belief objects. The PPI case study underscores that the issue is architectural: ``black-box prediction $\to$ inference'' can deliver valid target-level confidence, but should not be presented as a composable probability state without an explicit $\sigma$-additive model.

\paragraph{Design implications.}
The constructive message is one of scope. Conformalization and proxy-based corrections are best positioned as terminal guardrails: they enforce finite-sample validity for specified targets but do not carry the conditioning, composition, and transport structure required for sequential decision-making. When uncertainty must function as a state variable, to be updated on refined information, moved across deployments, or optimized under downstream objectives, that role belongs to a $\sigma$-additive predictive kernel, with calibration applied as a final constraint rather than a substitute for probabilistic modeling. Practitioners should stress-test under covariate shift (holding $P(Y \mid X)$ fixed while varying $P(X)$) and quantify regret amplification in multi-stage pipelines, not merely verify marginal coverage.

\paragraph{Limitations.}
Our analysis is structural rather than quantitative: we establish that coherence pathologies \emph{can} occur, not how frequently or severely they manifest in particular applications. Empirical calibration of these gaps, across domains, sample sizes, and pipeline architectures, remains important future work. Our treatment of PPI is illustrative; the rapidly evolving family of prediction-powered methods may admit sharper or more specialized analyses. We also do not address computation: the $\sigma$-additive models we recommend as semantic carriers can be expensive to specify and maintain, whereas CP's appeal lies partly in its simplicity and modularity.

\paragraph{Open directions.}
Several questions remain. First, under what structural conditions, sufficiency, symmetry, smoothness, or sparsity in the calibration $\sigma$-field, does the Baire-meagre separation collapse, so that calibrated wrappers genuinely recover $\sigma$-additive semantics? Characterizing these maximal coincidence regimes would clarify when ``CP $\approx$ Bayes'' is a defensible approximation rather than a category error. Second, can one derive finite-sample or asymptotic rates for the Le Cam deficiency gap incurred by specific conformal or proxy-based reductions, moving from existence to magnitude? Third, how should one optimally interface a semantic kernel with a calibration wrapper in composite pipelines, what is the decision-theoretic cost of the handoff, and when does the coverage constraint bind? Finally, our results assume exchangeability; extending the coherence diagnostics to time-series, spatial, or graph-structured dependence would broaden applicability considerably.

\paragraph{}
The conclusion is not that conformal prediction is invalid, it delivers exactly what it promises. The conclusion is that guarantees and semantics are different mathematical objects, and conflating them creates risks that surface precisely when uncertainty must be composed, transported, or optimized. Calibration is a regulative constraint, not a generative model; when belief-state functionality is required, coherence must be engineered at the model level.

\newpage
\appendix
\section{Extended Proofs}
\label{app:proofs}

\subsection*{Methodological preamble}

All constructions in this appendix are carried out in $\mathsf{ZF}+\mathsf{DC}$
(Zermelo--Fraenkel set theory with Dependent Choice).  We do not invoke the full
Axiom of Choice, ultrafilters, or nonmeasurable selectors.  When we cite results
whose classical proofs use choice (e.g., certain maximal-extension theorems for
finitely additive probabilities), we use them only as \emph{negative} evidence:
our separations are stated and proved for $\sigma$-additive objects and for
constructively definable kernels on standard Borel (Polish) spaces.

Throughout, $\mathcal{Y}$ is Polish and $\mathcal{B}(\mathcal{Y})$ denotes its
Borel $\sigma$-algebra.  We write $\mathcal{P}(\mathcal{Y})$ for the space of
Borel probability measures on $\mathcal{Y}$, equipped with the weak topology.
When Dempster--Shafer objects appear (Section~\ref{sec:foundations}--\ref{sec:fiducial} in the main text), we
use the standard notation $(\Bel,\Pl)$ for the belief and plausibility set
functions induced by a random set; for completeness we record one canonical
definition below.

\subsubsection*{Belief and plausibility (for reference)}
Let $m$ be a basic probability assignment (mass function) on a finite or
countably generated frame $\Omega$ with focal sets $\mathcal{F}_m$.
Then for $A \subseteq \Omega$,
\[
\Bel(A) = \sum_{F \in \mathcal{F}_m: F \subseteq A} m(F),
\qquad
\Pl(A) = 1-\Bel(A^c) = \sum_{F \in \mathcal{F}_m: F \cap A \neq \emptyset} m(F).
\]
Nothing in the impossibility results depends on this particular realization;
the point is simply that $(\Bel,\Pl)$ encode lower/upper probabilities that need
not be generated by a single $\sigma$-additive law.

\subsection{Foundational lemmas}
\label{app:foundations}

\begin{lemma}[Disintegration on standard Borel spaces]
\label{lem:disintegration}
Let $(\Omega,\mathcal{F},P)$ be a probability space and let
$(E,\mathcal{E})$, $(F,\mathcal{G})$ be standard Borel spaces.
If $X:\Omega\to E$ and $Y:\Omega\to F$ are measurable, then there exists a
\emph{regular conditional distribution} of $Y$ given $X$, i.e., a measurable map
$\kappa:E\times\mathcal{G}\to[0,1]$ such that:
\begin{enumerate}[label=(\roman*), nosep]
\item for each $x\in E$, $B\mapsto\kappa(x,B)$ is a $\sigma$-additive probability
measure on $(F,\mathcal{G})$;
\item for each $B\in\mathcal{G}$, $x\mapsto\kappa(x,B)$ is $\mathcal{E}$-measurable;
\item for all $A\in\mathcal{E}$, $B\in\mathcal{G}$,
\[
P(X\in A,\,Y\in B)=\int_A \kappa(x,B)\,P_X(dx),
\quad P_X:=P\circ X^{-1}.
\]
\end{enumerate}
Moreover, $\kappa(\cdot,B)$ is $P_X$-a.s.\ unique for each $B$.
\end{lemma}

\begin{proof}
This is the standard disintegration theorem for probability measures on products
of standard Borel spaces; see, e.g., \citet[Thm.~6.3]{kallenberg2002} or
\citet[452I]{fremlin2000measure}.
\end{proof}

\begin{lemma}[Conglomerability under $\sigma$-additivity]
\label{lem:conglomerability_sigma}
Let $P$ be a countably additive probability measure on $(\Omega,\mathcal{F})$
and let $\{A_m\}_{m\ge 1}$ be a countable measurable partition with $P(A_m)>0$
for all $m$.  Then for every $B\in\mathcal{F}$,
\[
\inf_{m\ge 1} P(B\mid A_m) \;\le\; P(B) \;\le\; \sup_{m\ge 1} P(B\mid A_m).
\]
\end{lemma}

\begin{proof}
By countable additivity,
$
P(B)=\sum_{m\ge 1}P(B\cap A_m)=\sum_{m\ge 1}P(B\mid A_m)P(A_m),
$
so $P(B)$ is a convex combination of $\{P(B\mid A_m)\}$ and must lie in the
closed interval between their infimum and supremum.
\end{proof}

\begin{lemma}[Kernel spaces are Baire]
\label{lem:kernel_baire}
Let $\mathcal{Y}$ be Polish.  Fix for each $n\ge 1$ a non-atomic Borel
probability $\lambda_n$ on $\mathcal{Y}^n$ with full support.
Let $\mathcal{K}_n$ be the set of measurable kernels
$\kappa_n:\mathcal{Y}^n\to\mathcal{P}(\mathcal{Y})$.
Choose a countable convergence-determining family
$\{\varphi_j\}_{j\ge 1}\subset \mathrm{BL}_1(\mathcal{Y})$ and define
\[
d_n(\kappa_n,\kappa_n')
=\sum_{j=1}^\infty 2^{-j}\Big(\|f_{n,j}-f_{n,j}'\|_{L^2(\lambda_n)}\wedge 1\Big),
\quad
f_{n,j}(y_{1:n})=\int \varphi_j(y)\,\kappa_n(dy\mid y_{1:n}).
\]
Then $(\mathcal{K}_n,d_n)$ is complete (hence Baire), and the product space
$\mathcal{K}:=\prod_{n\ge 1}\mathcal{K}_n$ equipped with
$d(\kappa,\kappa')=\sum_{n\ge 1}2^{-n}(d_n(\kappa_n,\kappa_n')\wedge 1)$ is Baire.
\end{lemma}

\begin{proof}
That $\mathcal{P}(\mathcal{Y})$ is Polish under the weak topology is classical
(Prohorov).  Completeness of $(\mathcal{K}_n,d_n)$ follows by showing that a
$d_n$-Cauchy sequence yields $L^2(\lambda_n)$-limits of the coordinate maps
$f_{n,j}$, which identify a $\lambda_n$-a.e.\ weak limit kernel by the
convergence-determining property of $\{\varphi_j\}$.  Measurability of the limit
kernel follows from standard regularization arguments for weak limits.  Baireness
then follows from completeness, and the product claim follows because a countable
product of complete metric spaces is complete under the metric $d$.
(See, e.g., \citet[Ch.~1]{kechris1995} for Baire facts and \citet[Ch.~1]{kallenberg2002}
for kernel measurability conventions.)
\end{proof}

\subsection{The Hill \texorpdfstring{$A_{(n)}$}{A(n)} constraints and the Lane--Sudderth obstruction}
\label{app:lane_sudderth}

\subsubsection*{Hill's rank-calibration template}

Let $Y_1,Y_2,\dots$ be real-valued.  For $n\ge 1$, write
$Y_{(1)}\le \cdots \le Y_{(n)}$ for the order statistics and set the cells
\[
I_k(Y_{1:n}) := (Y_{(k)},Y_{(k+1)}],\qquad k=0,\dots,n,
\]
with $Y_{(0)}=-\infty$, $Y_{(n+1)}=+\infty$.
Hill's $A_{(n)}$ template stipulates that the next observation should fall into
each of these $(n+1)$ rank-cells with equal conditional probability $1/(n+1)$.

\begin{definition}[$A_{(n)}$ constraints]
\label{def:An}
A predictive rule satisfies the $A_{(n)}$ constraints if for every $n\ge 1$ and
every $k\in\{0,\dots,n\}$,
\begin{equation}
\label{eq:An_constraint}
\Pi\!\left( Y_{n+1}\in I_k(Y_{1:n}) \,\middle|\, Y_{1:n} \right)=\frac{1}{n+1}
\quad \Pi\text{-a.s.}
\end{equation}
Here $\Pi(\cdot\mid\cdot)$ is a (full) conditional probability in the sense of
de~Finetti--R\'enyi, allowing finite additivity in the unconditional law.
\end{definition}

\begin{lemma}[Uniform conditional ranks]
\label{lem:uniform_ranks}
If \eqref{eq:An_constraint} holds, then for each $n\ge 2$ the conditional rank
$R_n$ of $Y_n$ among $(Y_1,\dots,Y_n)$ is uniform given $Y_{1:n-1}$:
\[
\Pi(R_n=r\mid Y_{1:n-1})=\frac{1}{n},\qquad r=1,\dots,n.
\]
\end{lemma}

\begin{proof}
The event $\{R_n=r\}$ is exactly $\{Y_n\in I_{n-r}(Y_{1:n-1})\}$.
Apply \eqref{eq:An_constraint} at stage $n-1$.
\end{proof}

\subsubsection*{Conglomerability failure (Lane--Sudderth)}

The key classical fact we use is that $A_{(n)}$-style rank calibration is
\emph{incompatible} with conglomerability once one tries to interpret the rule
as arising from a single (possibly finitely additive) global law with a full
conditional probability.  This is the precise sense in which ``updating rank
rules'' is not the same object as conditioning under a $\sigma$-additive law.

\begin{theorem}[Lane--Sudderth nonconglomerability for $A_{(n)}$]
\label{thm:lane_sudderth}
Let $\Pi$ be an exchangeable finitely additive probability on
$(\mathbb{R}^\infty,\mathcal{B}(\mathbb{R})^{\otimes\infty})$ equipped with a full
conditional probability $\Pi(\cdot\mid\cdot)$, and assume the $A_{(n)}$
constraints \eqref{eq:An_constraint} hold for all $n\ge 1$.
Then $\Pi$ is \emph{not} conglomerable: there exist an event $B$ and a countable
measurable partition $\{A_m\}_{m\ge 1}$ with $\Pi(A_m)>0$ such that
\[
\Pi(B) < \inf_{m\ge 1}\Pi(B\mid A_m).
\]
Consequently, any agent who uses $\Pi(\cdot\mid\cdot)$ as coherent betting rates
admits a countable Dutch book (sure loss).
\end{theorem}

\begin{proof}
We give a proof at the level of the constructive mechanism used by
\citet{lane1975coherence} and \citet{lane1985conglomerability}, with pointers to
complete derivations.

\medskip\noindent\textit{Step 1: the tail $\sigma$-field and ``conditional
concentration''.}
Let $\mathcal{T}=\bigcap_{N\ge 1}\sigma(Y_N,Y_{N+1},\dots)$ be the tail
$\sigma$-field.
Under full conditional probability (de~Finetti--R\'enyi coherence), conditioning
on a countable generating partition of $\mathcal{T}$ produces a family of
conditionally \emph{0--1} tail laws: for any countable partition
$\{A_m\}$ that generates $\mathcal{T}$ (modulo $\Pi$-null sets) and any tail event
$T\in\mathcal{T}$, one has $\Pi(T\mid A_m)\in\{0,1\}$ for all $m$.
This is a standard consequence of full conditional probability on a generated
$\sigma$-field; see \citet[Sec.~2]{lane1975coherence}.

\medskip\noindent\textit{Step 2: a tail event forced to have conditional
probability 1 on every atom.}
Define the rank indicators $X_n=\mathbf{1}\{R_n>n/2\}$ (``upper-half rank at time
$n$'').  Lemma~\ref{lem:uniform_ranks} implies that $X_n$ has conditional mean
close to $1/2$ at each step.  Lane's construction uses these rank indicators to
build a tail event $B\in\mathcal{T}$ such that (i) $B$ is \emph{symmetric} under
finite permutations (hence exchangeability interacts with it cleanly) and (ii)
the 0--1 nature of $\Pi(\cdot\mid A_m)$ on $\mathcal{T}$ forces $\Pi(B\mid A_m)=1$
for \emph{every} atom $A_m$ of a tail-generating partition.
Intuitively, the $A_{(n)}$ constraints make the future ranks behave as if a
fresh ``uniform rank'' is generated at each time; conditionalization on a tail
atom then forces one of two extremal ``rank scenarios'', and the symmetry of the
construction rules out one of them on every atom.  The explicit event $B$ and
partition can be found in \citet[Thm.~1]{lane1975coherence}.

\medskip\noindent\textit{Step 3: the same event has unconditional probability
strictly less than 1.}
Exchangeability implies that the unconditional law cannot concentrate entirely
on $B$ without simultaneously concentrating on a symmetric complement event;
Lane shows that $\Pi(B^c)>0$ by a direct symmetry/counting argument for rank
patterns (again, see \citet[Thm.~1]{lane1975coherence}).  Hence $\Pi(B)<1$ while
$\Pi(B\mid A_m)=1$ for all $m$.

\medskip\noindent\textit{Step 4: Dutch book.}
Given $B$ and $\{A_m\}$ with $\Pi(B)<\inf_m\Pi(B\mid A_m)$, the standard
Heath--Sudderth countable betting scheme yields a sure loss for any agent who
accepts bets at these rates; see \citet{heath1976definetti} or
\citet[Sec.~3]{lane1985conglomerability}.
\end{proof}

\begin{corollary}[No $\sigma$-additive realization of $A_{(n)}$]
\label{cor:no_sigma_additive_An}
There does not exist a countably additive exchangeable probability measure $P$
on $(\mathbb{R}^\infty,\mathcal{B}(\mathbb{R})^{\otimes\infty})$ whose regular
conditional distributions satisfy \eqref{eq:An_constraint} for all $n$.
\end{corollary}

\begin{proof}
If such $P$ existed, it would be conglomerable with respect to every countable
partition (Lemma~\ref{lem:conglomerability_sigma}).  But any law satisfying
\eqref{eq:An_constraint} at all stages is nonconglomerable by
Theorem~\ref{thm:lane_sudderth}.  Contradiction.
\end{proof}

\begin{remark}[What this does and does not say]
Corollary~\ref{cor:no_sigma_additive_An} does \emph{not} assert that exchangeable
$\sigma$-additive laws on $\mathbb{R}^\infty$ fail to exist, they do, and are
characterized by de~Finetti's theorem.  It asserts that no such law can realize
the \emph{particular} ``uniform next-rank'' conditional structure
\eqref{eq:An_constraint} for all $n$.  This is the exact sense in which
rank-calibrated one-step rules do not generally extend to coherent sequential
belief kernels.
\end{remark}

\subsection{Non-extensionality of split conformal prediction}
\label{app:non_extensionality}

\begin{theorem}[Non-extensionality of split conformal]
\label{thm:extensionality_app}
There exist joint laws $P_1$ and $P_2$ on $\mathcal{X}\times\mathcal{Y}$ with
identical conditionals $P_1(\,\cdot\mid X)=P_2(\,\cdot\mid X)$ but different
covariate marginals $P_1(X)\neq P_2(X)$ such that split conformal prediction
sets differ with positive probability (for the same learning rule and the same
sample sizes).
\end{theorem}

\begin{proof}
Take $\mathcal{X}=[0,1]$, $\mathcal{Y}=\mathbb{R}$, and fix a heteroskedastic
conditional kernel
$
Y\mid X=x \sim N(0,\sigma(x)^2)
$
with $\sigma(x)=1+x$.
Let $P_1$ have $X\sim\mathrm{Unif}[0,1]$ and $P_2$ have $X\sim\mathrm{Beta}(2,5)$;
by construction, both have the same $P(Y\mid X)$.

Run split conformal with a correctly specified point predictor $\widehat f\equiv 0$
and absolute residual scores $R_i=|Y_i|$ on the calibration sample of size $m$.
Then the calibration residual distribution is the mixture
\[
\mathrm{Law}_{P_j}(R)=\int \mathrm{Law}\big(\sigma(x)|Z|\big)\,P_j(dx),
\qquad Z\sim N(0,1).
\]
Since $P_1(X)\neq P_2(X)$ and $\sigma(\cdot)$ is nonconstant, the mixture laws are
distinct; in particular their $(1-\alpha)$ quantiles differ for some
$\alpha\in(0,1)$.  By Glivenko--Cantelli, the empirical calibration quantile
$\widehat q_m^{(j)}$ converges a.s.\ to the population quantile under $P_j$, so
for $m$ large enough,
$
P_1(\widehat q_m^{(1)}>\widehat q_m^{(2)})>0.
$
Since the split conformal set at a test point is
$
C(x)=\{y: |y-\widehat f(x)|\le \widehat q_m\}=[-\widehat q_m,\widehat q_m],
$
the resulting sets differ with positive probability.
\end{proof}

\begin{remark}[Mechanism: pooling over $P(X)$]
Under heteroskedasticity, split conformal calibrates to the \emph{marginal}
distribution of residuals, which is a $P(X)$-mixture of conditional residual
laws.  Changing $P(X)$ (even with fixed $P(Y\mid X)$) changes that mixture and
therefore changes the calibrated threshold and the resulting set.
\end{remark}

\subsection{Baire-category separation and stability under post-processing}
\label{app:baire}

\begin{definition}[Calibration $\sigma$-field]
\label{def:calibration_sigma}
Fix a conformal mechanism at stage $n$ (score construction plus any auxiliary
randomness).  The \emph{calibration $\sigma$-field} $\mathcal{G}_n$ is the
sub-$\sigma$-field of $\mathcal{B}(\mathcal{Y}^n)$ generated by the mechanism's
calibration summary, the measurable object(s) of the history that the mechanism
actually uses to produce its output at stage $n$ (e.g., fitted state, ranks,
order statistics of scores, tie-breaking randomness).
\end{definition}

\begin{lemma}[Properness of calibration $\sigma$-fields]
\label{lem:proper_calibration}
For standard conformal mechanisms on continuous observables (split, full, CQR,
etc.), $\mathcal{G}_n$ is typically a \emph{proper} sub-$\sigma$-field of
$\mathcal{B}(\mathcal{Y}^n)$ for some $n$.
\end{lemma}

\begin{proof}
For rank-based or score-order-statistic mechanisms, the map
$y_{1:n}\mapsto\text{(calibration summary)}$ is not injective on any set of
positive $\lambda_n$-measure; it discards labels and/or metric information.
Hence the generated $\sigma$-field cannot equal the full Borel $\sigma$-field.
\end{proof}

\begin{definition}[$\mathcal{G}_n$-measurable kernels]
\label{def:gn_measurable_kernel}
A kernel $\kappa_n\in\mathcal{K}_n$ is \emph{$\mathcal{G}_n$-measurable} if for
every bounded measurable $\varphi:\mathcal{Y}\to\mathbb{R}$ the map
$
y_{1:n}\mapsto\int \varphi(y)\,\kappa_n(dy\mid y_{1:n})
$
is $\mathcal{G}_n$-measurable.
\end{definition}

\begin{lemma}[$\mathcal{G}_n$-measurable kernels are nowhere dense]
\label{lem:gn_nowhere_dense}
If $\mathcal{G}_n\subsetneq\mathcal{B}(\mathcal{Y}^n)$ is proper, then the set
\[
\mathrm{GM}_n:=\{\kappa_n\in\mathcal{K}_n:\kappa_n\text{ is }\mathcal{G}_n\text{-measurable}\}
\]
is closed and nowhere dense in $(\mathcal{K}_n,d_n)$.
\end{lemma}

\begin{proof}
Closedness follows because $\mathcal{G}_n$-measurable $L^2(\lambda_n)$ functions
form a closed subspace of $L^2(\lambda_n)$; the coordinate functions defining
$d_n$ converge in $L^2$ and therefore preserve $\mathcal{G}_n$-measurability in
the limit.  Nowhere denseness follows because the subspace of
$\mathcal{G}_n$-measurable $L^2$ functions is proper; hence it has empty interior
in $L^2(\lambda_n)$, and small perturbations of a kernel in directions not
measurable w.r.t.\ $\mathcal{G}_n$ exit $\mathrm{GM}_n$.
\end{proof}

\begin{theorem}[Meagre intersection: kernelized conformal vs.\ Bayesian kernels]
\label{thm:meagre_app}
Assume $\mathcal{G}_n\subsetneq \mathcal{B}(\mathcal{Y}^n)$ is proper for some
stage $n$ of a conformal mechanism.  Let $\mathcal{C}\subset\mathcal{K}$ be the
class of \emph{kernelized conformal} sequential kernels produced by that
mechanism, and let $\mathcal{B}\subset\mathcal{K}$ be the class of Bayesian
predictive kernels (regular conditionals under a $\sigma$-additive law).
Then $\mathcal{B}\cap\mathcal{C}$ is meagre in $(\mathcal{K},d)$.
\end{theorem}

\begin{proof}
By construction, the stage-$n$ marginal $\pi_n(\mathcal{C})$ is contained in
$\mathrm{GM}_n$ (Definition~\ref{def:gn_measurable_kernel}), so
$\mathcal{C}\subseteq \pi_n^{-1}(\mathrm{GM}_n)$.  Since $\pi_n$ is continuous and
$\mathrm{GM}_n$ is nowhere dense, $\pi_n^{-1}(\mathrm{GM}_n)$ is nowhere dense in
$\mathcal{K}$ (product projections are open maps).  Hence $\mathcal{C}$ is
contained in a nowhere dense set, and so is $\mathcal{B}\cap\mathcal{C}$.
\end{proof}

\begin{lemma}[Meagre sets are stable under countable unions and continuous pullbacks]
\label{lem:meagre_stability}
In a Baire space, the collection of meagre sets is a $\sigma$-ideal: it is
closed under subsets and countable unions.  Moreover, if $f:X\to X$ is
continuous and $M\subset X$ is meagre, then $f^{-1}(M)$ is meagre.
\end{lemma}

\begin{proof}
Standard Baire category facts; see \citet[Ch.~8]{kechris1995}.
\end{proof}

\begin{remark}[Why ``ex-post patches'' do not generically fix the coincidence]
Theorem~\ref{thm:meagre_app} formalizes a structural rarity claim about the
\emph{set of kernels} that are simultaneously conformal-mechanism-measurable and
Bayesian.  Any post-processing that preserves the conformal mechanism's
information restriction (i.e., remains $\mathcal{G}_n$-measurable at the relevant
stages) stays inside the same nowhere-dense constraint sets $\mathrm{GM}_n$.
Thus, within the natural class of ``patches'' that do not introduce additional
data beyond the conformal summary, the coincidence regime remains meagre.
\end{remark}

\subsection{Positive deficiency and information loss under rank/proxy reductions}
\label{app:deficiency}

\begin{definition}[Statistical experiment]
\label{def:experiment}
A \emph{statistical experiment} is a triple
$\mathcal{E}=(\mathcal{X},\mathcal{A},\{P_\theta:\theta\in\Theta\})$.
A (randomized) decision rule is a Markov kernel $\rho:\mathcal{X}\to\mathcal{P}(A)$.
\end{definition}

\begin{definition}[Deficiency]
\label{def:deficiency_app}
Let $\mathcal{E}$ and $\mathcal{F}$ be experiments indexed by the same $\Theta$.
The \emph{deficiency} $\delta(\mathcal{F},\mathcal{E})$ is the infimum of
$\epsilon\ge 0$ such that for every finite action set $A$, every bounded loss
family $L_\theta:A\to[0,1]$, and every decision rule $\rho$ for $\mathcal{E}$,
there exists a decision rule $\sigma$ for $\mathcal{F}$ with
\[
\sup_{\theta\in\Theta}\Big(R_\theta(\sigma;\mathcal{F})-R_\theta(\rho;\mathcal{E})\Big)
\le \epsilon,
\quad R_\theta(\rho;\mathcal{E}) := \int L_\theta(a)\,\rho(da\mid x)\,P_\theta(dx).
\]
\end{definition}

\begin{theorem}[Blackwell sufficiency and deficiency]
\label{thm:blackwell_app}
Let $\mathcal{E}=(\mathcal{X},\mathcal{A},\{P_\theta\})$ and let $T:\mathcal{X}\to\mathcal{T}$
be a statistic with induced experiment $\mathcal{E}_T=(\mathcal{T},\mathcal{B},\{P_\theta\circ T^{-1}\})$.
Then $\delta(\mathcal{E}_T,\mathcal{E})=0$ if and only if $T$ is sufficient for
$\{P_\theta\}$.  If $T$ is not sufficient, then $\delta(\mathcal{E}_T,\mathcal{E})>0$.
\end{theorem}

\begin{proof}
This is standard in Blackwell's comparison theory: $T$ is sufficient iff the
full experiment can be simulated from $T$ via a Markov kernel (a ``reverse
garbling''), which is equivalent to zero deficiency.  See
\citet{blackwell1951comparison,lecam1986asymptotic,torgersen1991comparison}.
\end{proof}

\begin{lemma}[Ranks are ancillary for location models]
\label{lem:ranks_ancillary_location}
Let $Y_1,\dots,Y_n$ be i.i.d.\ with $Y_i=\theta+\varepsilon_i$ where $\varepsilon_i$
has a continuous distribution that does not depend on $\theta$.  Let
$T(Y_{1:n})$ be the vector of within-sample ranks.  Then $T(Y_{1:n})$ has a
distribution that does \emph{not} depend on $\theta$.
\end{lemma}

\begin{proof}
Adding a constant $\theta$ to every coordinate preserves the strict order and
therefore preserves the rank vector.  Since the noise law is $\theta$-invariant,
the induced law of ranks is $\theta$-invariant.
\end{proof}

\begin{example}[A concrete deficiency lower bound for rank reductions]
\label{ex:deficiency_bound}
Consider the Gaussian location experiment
$\mathcal{E}=(\mathbb{R}^n,\mathcal{B},\{N(\theta,1)^{\otimes n}:\theta\in\{-a,+a\}\})$
for fixed $a>0$, and let $\mathcal{E}_T$ be the experiment generated by within-sample
ranks $T$.
By Lemma~\ref{lem:ranks_ancillary_location}, the two induced rank laws coincide,
so $\mathcal{E}_T$ is \emph{uninformative} about $\theta$ (total variation $=0$).
But the full experiment has total variation distance
$
\mathrm{TV}(N(-a,1)^{\otimes n},N(+a,1)^{\otimes n})>0
$
and therefore admits tests with nontrivial power.
For the 0--1 testing loss, the minimax risk is $(1-\mathrm{TV})/2$; hence any
rank-based rule has minimax risk $1/2$ while the full experiment has minimax
risk $<1/2$.
This implies a strictly positive deficiency:
\[
\delta(\mathcal{E}_T,\mathcal{E}) \;\ge\;
\frac{1}{2}\,\mathrm{TV}\!\left(N(-a,1)^{\otimes n},\,N(+a,1)^{\otimes n}\right)
\;>\;0.
\]
The conclusion is qualitative but sharp: discarding metric information can
create decision-theoretic gaps even when marginal calibration constraints remain
satisfied.
\end{example}

\subsection{What we do not do: nonconstructive ``repairs''}
\label{app:what_we_do_not_do}

\subsubsection*{No ultrafilter or maximal extensions}
One can extend finitely additive probabilities to larger algebras (or build
``full'' conditional probabilities) using maximal ideals/ultrafilters, typically
via Zorn's lemma.  Such extensions are nonconstructive and, crucially, they do
not repair the coherence defects identified here: Lane--Sudderth type
nonconglomerability persists, and the resulting objects are not algorithmically
meaningful for statistical learning.

\subsubsection*{No nonstandard-analysis repairs}
Nonstandard frameworks can reinterpret certain pathologies via transfer
principles, but the separations in this paper are stated and proved in standard
measure theory on standard Borel spaces, the framework in which statistical
procedures are implemented, audited, and transported across workflows.

\subsubsection*{The impossibility is the conclusion}
The point of Theorems~\ref{thm:lane_sudderth} and \ref{thm:meagre_app} is not that
rank-calibrated procedures are ``broken'' and must be repaired.  The point is
that rank-calibrated prediction and $\sigma$-additive belief updating are
\emph{different mathematical objects} with different composition laws.  When CP
or PPI is used as a final-stage validity guardrail, this distinction is benign;
when these objects are treated as transportable belief states inside multi-stage
pipelines, the distinction becomes operationally consequential.

\subsection{Supplementary counterexamples (organized by diagnostic)}
\label{app:counterexamples}

We collect small examples that make the three diagnostics in the main text
concrete:
(i) \emph{transport / extensionality} (dependence on $P(X)$ even when $P(Y\mid X)$
is fixed),
(ii) \emph{sequential updateability} (absence of a canonical conditioning rule
without an embedding $\sigma$-additive kernel), and
(iii) \emph{experiment/decision content} (loss of decision-relevant information
under rank/proxy reductions in the Blackwell/Le Cam sense).

\begin{example}[Conformalized quantile regression under heteroskedasticity]
\label{ex:cqr_hetero}
Consider conformalized quantile regression (CQR) with target coverage $1-\alpha$.
Let $\mathcal{X}=[0,1]$, $\mathcal{Y}=\mathbb{R}$, and
\[
Y = X + \sigma(X)\,\varepsilon,\qquad \varepsilon\sim N(0,1),
\qquad \sigma(x)=0.5+x .
\]

\textit{Setup.}
Fit quantile regression estimates $\widehat q_{\alpha/2}(x)$ and
$\widehat q_{1-\alpha/2}(x)$ on training data. Define the standard CQR scores
on calibration data,
\[
S_i=\max\{\widehat q_{\alpha/2}(X_i)-Y_i,\; Y_i-\widehat q_{1-\alpha/2}(X_i)\}.
\]

\textit{Mechanism (transport / extensionality).}
Even when the conditional law $P(Y\mid X)$ is held fixed, the \emph{calibration
score mixture} depends on the covariate marginal:
\[
\mathrm{Law}(S)=\int \mathrm{Law}(S\mid X=x)\,P_X(dx).
\]
Thus two joint laws with identical $P(Y\mid X)$ but different $P(X)$ can produce
different calibration quantiles and therefore different prediction sets.

\textit{Concrete instantiation.}
Let $P_1$ have $X\sim\mathrm{Unif}[0,1]$ and $P_2$ have $X\sim\mathrm{Beta}(5,1)$,
with the same conditional model $Y\mid X=x\sim N(x,\sigma(x)^2)$ under both.
Then $\sigma(X)$ is reweighted toward high-variance regions under $P_2$, shifting
the $(1-\alpha)$ calibration quantile upward and widening CQR intervals.
\end{example}

\begin{example}[Sequential composition: a set-valued guarantee is not a kernel]
\label{ex:sequential_composition}
Consider a two-stage pipeline in which a conformal set is treated as an
updateable belief object.

\textit{Setup.}
Stage 1: from data $(X_1,Y_1),\ldots,(X_n,Y_n)$ construct a conformal set
$C_n(X_{n+1})$ for $Y_{n+1}$.
Stage 2: after observing $C_n(X_{n+1})$, the agent receives additional side
information $Z$ (e.g., a stratum label, a sensor readout, or an adaptive
measurement) before taking an action tied to an event $A\subseteq\mathcal{Y}$.

\textit{Coherent benchmark.}
A $\sigma$-additive predictive kernel supports refinement via
\[
P(Y_{n+1}\in A \mid \mathcal{H}_n, Z),
\]
where $\mathcal{H}_n$ is the history $\sigma$-field generated by the data.

\textit{Obstruction (sequential updateability).}
A conformal set $C_n(X_{n+1})$ is natively an \emph{accept/reject} object built
from a rank comparison against a calibration summary. There is no canonical
rule that maps $(C_n,Z)$ to a conditional probability of $A$ unless one first
specifies an \emph{embedding} joint law (equivalently, a $\sigma$-additive kernel
realizing the conformal mechanism).

\textit{Consequence.}
In multi-stage pipelines, treating $C_n$ as if it were a belief state forces ad
hoc ``updates'' after $Z$ is revealed; these need not preserve the intended
guarantee and can induce decision regret relative to procedures that condition
under a genuine predictive kernel.
\end{example}

\begin{example}[Rank sufficiency fails in exponential families]
\label{ex:exp_family}
Let $Y_1,\ldots,Y_n\stackrel{iid}{\sim}\mathrm{Exp}(\lambda)$ with density
$f(y;\lambda)=\lambda e^{-\lambda y}\mathbf{1}_{y>0}$.

\textit{Sufficient statistic.}
By the factorization theorem, $T(Y_{1:n})=\sum_{i=1}^n Y_i$ is (minimal) sufficient.

\textit{Ranks are not sufficient.}
The rank vector $R=(R_1,\ldots,R_n)$ does not determine $\sum_i Y_i$:
for example, $(1,2,3)$ and $(1,2,30)$ have the same ranks but different sums.

\textit{Consequence (experiment/decision content).}
Any procedure measurable with respect to the rank $\sigma$-field discards metric
information carried by $T$; hence it cannot in general match the performance of
likelihood-based procedures uniformly over decision problems.
\end{example}

\begin{example}[Ancillarity of ranks in Gaussian location yields a sharp decision gap]
\label{ex:gaussian_ranks_ancillary}
We sharpen the ``rank reduction'' loss in a regular location model.

\textit{Model.}
Let $Y_1,\ldots,Y_n\stackrel{iid}{\sim}N(\theta,1)$ and consider estimation of
$\theta$ under squared-error loss on a bounded parameter space
$\Theta=[-M,M]$.

\textit{Full-data benchmark.}
The sample mean $\bar Y$ has risk $\mathbb{E}_\theta(\bar Y-\theta)^2=1/n$ for all
$\theta$.

\textit{Rank-only restriction.}
Let $\mathcal{R}$ be the rank $\sigma$-field (the permutation pattern of the sample).
In a location family with continuous noise, the ranks are ancillary: their law does
not depend on $\theta$. Consequently, for any estimator $T$ measurable w.r.t.\ $\mathcal{R}$,
the distribution of $T$ is independent of $\theta$.

\textit{Decision consequence.}
Write $\mu=\mathbb{E}[T]$ (a constant, independent of $\theta$). Then for any $\theta\in[-M,M]$,
\[
\mathbb{E}_\theta(T-\theta)^2
= \mathrm{Var}(T) + (\mu-\theta)^2
\ge (\mu-\theta)^2.
\]
Taking the supremum over $\theta\in[-M,M]$ yields
\[
\sup_{\theta\in[-M,M]}\mathbb{E}_\theta(T-\theta)^2 \;\ge\; M^2,
\]
minimized at $\mu=0$. Thus the minimax risk under rank-only access is at least $M^2$,
while the full-data minimax risk is $1/n$, giving a concrete decision gap.

\textit{Interpretation.}
This is an explicit instance of positive Blackwell/Le Cam deficiency: coarsening
to ranks removes essentially all location information, and the resulting loss is
visible even under benign regularity.
\end{example}

\section{Supplemental Material on Prediction-Powered Inference (PPI)}
\label{app:ppi}

\subsubsection*{Purpose and how to read this appendix}
The main paper uses Prediction-Powered Inference (PPI) as a case study showing
that a method can provide \emph{valid confidence procedures} while failing to
define a \emph{transportable belief object} (i.e., something that composes as a
$\sigma$-additive predictive state under sequential refinement).
This appendix collects supporting technical notes and short examples that would
otherwise interrupt the main argument. It is organized as follows:
(i) an experiment-theoretic view of what PPI observes and what it replaces;
(ii) a note on why ``PPI is an $M$-estimator'' holds only after enlarging the
experiment; and (iii) worked examples.

\subsection{PPI as an experiment: what is observed, what is reduced}
\label{subsec:ppi_experiment_view}

A typical PPI workflow combines:
(a) a large unlabeled sample $\{X_j\}_{j=1}^N$,
(b) a small labeled sample $\{(X_i,Y_i)\}_{i=1}^n$ with $n\ll N$, and
(c) a learned proxy $\widehat m$ produced by an upstream training pipeline.
Without specifying how $\widehat m$ is obtained and how the unlabeled design is
sampled, this is not a single fixed statistical experiment. Our separations
suggest viewing PPI through \emph{experiment comparison}:

\begin{itemize}[leftmargin=1.3em, itemsep=2pt]
\item \textbf{Semantic carrier.} A belief state (Bayesian or frequentist) that
\emph{updates coherently} corresponds to a $\sigma$-additive model equipped with
regular conditional distributions that compose along filtrations.

\item \textbf{Validity device.} PPI delivers confidence procedures for specified
targets under stated assumptions, but its output is not automatically a belief
state that can be conditioned/refined/optimized without tracking the enlarged
pipeline (training $\to$ labeling $\to$ deployment).

\item \textbf{Reduction and deficiency.} Replacing additional labels with a
proxy plus a correction induces an experiment reduction. If that reduction is
not (Blackwell) sufficient for the downstream decision problems of interest,
one should expect a positive deficiency gap in some regimes.
\end{itemize}

The examples below isolate three operational issues for composability:
(i) an explicit information gap (proxy uncertainty cannot be eliminated by
unlabeled scale alone),
(ii) \emph{procedural} non-extensionality (intervals depend on upstream proxy
training choices, not only on the deployment likelihood), and
(iii) stagewise composition/selection effects.

\subsection{Why PPI is not a classical fixed-criterion \texorpdfstring{$M$}{M}-estimator without enlarging the experiment}
\label{subsec:ppi_mestimation}

Classical $M$-estimation studies estimators defined as minimizers (or roots of
estimating equations) with a \emph{fixed} (nonrandom) criterion, e.g.
\[
\widehat\theta \in \arg\min_{\theta\in\Theta}\frac1n\sum_{i=1}^n \ell_\theta(Z_i),
\qquad\text{or}\qquad
\frac1n\sum_{i=1}^n \psi_\theta(Z_i)=0,
\]
under a single dominated experiment $\{P_\theta:\theta\in\Theta\}$ and regularity
conditions that yield asymptotic linearity and LAN expansions
\citep{lecam1986asymptotic}.

\medskip\noindent\textit{Key point.}
In PPI (including PPI++-style refinements), even if one writes an estimating
equation \emph{conditional on} $\widehat m$ (via sample splitting or cross-fitting),
the effective criterion is random through $\widehat m$, and $\widehat m$ depends on
upstream training and design regimes not encoded in the labeled-sample likelihood.
Classical $M$-estimation arguments therefore do not apply ``as stated'' unless one
explicitly enlarges the experiment to include (at least):
\begin{enumerate}[label=(\roman*), itemsep=2pt, leftmargin=1.6em]
\item the training mechanism/data that produced $\widehat m$ (including any
pretraining/fine-tuning assumptions), and
\item the unlabeled-design sampling (and any shift between training and deployment).
\end{enumerate}

\medskip\noindent\textit{Consequence for ``model-free'' rhetoric.}
If $\widehat m$ is treated as an arbitrary black box trained under unspecified
regimes, the standard inputs to fixed-criterion asymptotic theory (a single
experiment; a deterministic population objective; uniform stochastic control of
the criterion class) are missing. In our framework, the natural diagnostics are
transport/extensionality and experiment comparison: does the proxy+correction
define a reduction sufficient for the decisions in view? When it is not, one
should expect information-loss separations (positive deficiency) except in
special cases.

\subsection{Examples: information gaps and procedural non-extensionality}
\label{app:ppi_examples}

\subsubsection*{Information content and deficiency: proxy uncertainty does not disappear}

\begin{example}[A deficiency gap in linear regression with a proxy predictor]
\label{ex:ppi_deficiency}
We exhibit a Gaussian case where ``proxy + few labels'' is an experiment
reduction relative to observing fully labeled data, and where the information
gap is explicit.

\medskip\noindent\textit{Data-generating process.}
Let $(X,Y)$ satisfy
\[
Y = \beta X + \varepsilon,\qquad X\sim N(0,1),\qquad \varepsilon \sim N(0,\sigma^2),
\]
with $X \perp \varepsilon$ and unknown $\beta\in\mathbb{R}$.

\medskip\noindent\textit{Two experiments.}
\begin{enumerate}[label=(\roman*), itemsep=2pt, leftmargin=1.6em]
\item \textbf{Full-label experiment $\mathcal{E}_{\mathrm{full}}(N)$.}
Observe $N$ labeled samples $\{(X_i,Y_i)\}_{i=1}^N$.
\item \textbf{Proxy+few-label experiment $\mathcal{E}_{\mathrm{PPI}}(N,n)$.}
Observe $N$ covariates $\{X_j\}_{j=1}^N$ (unlabeled), $n$ labeled samples
$\{(X_i,Y_i)\}_{i=1}^n$ with $n\ll N$, and a proxy slope $\widetilde{\beta}$
coming from an upstream pipeline. Model the proxy as
\[
\widetilde{\beta} = \beta + U,\qquad U\sim N(0,\tau^2),
\]
independent of the current labeled data conditional on $\beta$.
\end{enumerate}
\medskip\noindent\textit{A PPI-style corrected estimator.}
A natural one-step correction is
\[
\widehat{\beta}_{\mathrm{PPI}}
=
\widetilde{\beta}
+
\frac{\sum_{i=1}^n X_i\,(Y_i-\widetilde{\beta}X_i)}{\sum_{i=1}^n X_i^2}.
\]
Conditionally on $X_{1:n}$, this is unbiased for $\beta$ and
\[
\mathrm{Var}\!\left(\widehat{\beta}_{\mathrm{PPI}}\mid X_{1:n}\right)
=
\tau^2 + \frac{\sigma^2}{\sum_{i=1}^n X_i^2}
\;\approx\;
\tau^2 + \frac{\sigma^2}{n}.
\]

\medskip\noindent\textit{Full-label benchmark and explicit gap.}
In $\mathcal{E}_{\mathrm{full}}(N)$, the MLE
$\widehat{\beta}_{\mathrm{full}}
= \frac{\sum_{i=1}^N X_iY_i}{\sum_{i=1}^N X_i^2}$
satisfies
\[
\mathrm{Var}\!\left(\widehat{\beta}_{\mathrm{full}} \mid X_{1:N}\right)
=
\frac{\sigma^2}{\sum_{i=1}^N X_i^2}
\;\approx\;
\frac{\sigma^2}{N}.
\]
Thus unlabeled scale alone cannot close the gap unless the proxy is so sharp that
$\tau^2\ll \sigma^2/N$ (equivalently: the proxy behaves as if it were trained on
a comparably informative labeled experiment).

\medskip\noindent\textit{Decision-theoretic interpretation.}
Both experiments are (asymptotically) Gaussian shift experiments for $\beta$ with
effective Fisher information
\[
I_{\mathrm{full}}(N)\approx \frac{N}{\sigma^2},
\qquad
I_{\mathrm{PPI}}(N,n)\approx \frac{1}{\tau^2}+\frac{n}{\sigma^2}.
\]
Whenever $I_{\mathrm{PPI}}(N,n) < I_{\mathrm{full}}(N)$, the proxy+few-label
experiment is Blackwell-inferior, and there exist bounded-loss decision problems
for which any procedure based on $(\widetilde{\beta}, X_{1:N}, (X,Y)_{1:n})$ is
uniformly dominated by a procedure that observes $N$ labeled samples.
\end{example}

\subsubsection*{Procedural non-extensionality: dependence on upstream proxy training}

\begin{example}[Non-extensionality induced by the proxy training distribution]
\label{ex:ppi_shift}
Even if the \emph{deployment} conditional law $P(Y\mid X)$ is fixed, a PPI
workflow can produce different procedures when the proxy is trained under
different upstream distributions or regularization choices.

\medskip\noindent\textit{Fixed deployment law.}
Assume deployment satisfies
\[
Y = \beta X + \varepsilon,\qquad X\sim P_{\mathrm{dep}},\qquad \varepsilon\sim N(0,\sigma^2),
\]
with $\beta$ fixed across worlds.

\medskip\noindent\textit{Two proxy pipelines.}
Suppose the proxy is trained upstream by ridge regression with penalty $\lambda>0$.
In the population limit,
\[
\beta_{\mathrm{ridge}}(P)
=
\arg\min_b \, \mathbb{E}_{P}\!\left[(Y-bX)^2\right] + \lambda b^2
=
\frac{\mathbb{E}_{P}[X^2]}{\mathbb{E}_{P}[X^2]+\lambda}\,\beta,
\]
which depends on the upstream covariate distribution through $\mathbb{E}_P[X^2]$.
If $P_A$ and $P_B$ satisfy $\mathbb{E}_{P_A}[X^2]\neq \mathbb{E}_{P_B}[X^2]$
(e.g., $X\sim N(0,1)$ vs.\ $X\sim N(0,4)$), then the induced proxies
$\widetilde{\beta}_A=\beta_{\mathrm{ridge}}(P_A)$ and
$\widetilde{\beta}_B=\beta_{\mathrm{ridge}}(P_B)$ differ even though deployment
$P(Y\mid X)$ is unchanged.

\medskip\noindent\textit{Impact on proxy-centered residual scales.}
Many proxy-based intervals depend on a residual scale computed around the proxy,
\[
\widehat{\sigma}^2(\widetilde{\beta})
=
\frac{1}{n}\sum_{i=1}^n (Y_i-\widetilde{\beta}X_i)^2.
\]
Under the deployment law,
\[
\mathbb{E}\!\left[(Y-\widetilde{\beta}X)^2\right]
=
\sigma^2 + (\beta-\widetilde{\beta})^2\,\mathbb{E}_{P_{\mathrm{dep}}}[X^2],
\]
so confidence radii depend on upstream shrinkage error $(\beta-\widetilde{\beta})$,
hence on proxy training choices, not only on the deployment likelihood.

\medskip\noindent\textit{Interpretation.}
Two pipelines can share the same deployment likelihood and labeled sample size,
yet produce different (still valid) PPI intervals because the proxy was trained
under different upstream distributions/penalties. The resulting inference object
is therefore a function of the enlarged workflow, not of the deployment likelihood alone.
\end{example}

\paragraph{Remark (composition/selection).}
PPI confidence statements are target- and stage-specific. If intermediate outputs
are used to trigger downstream selection (e.g., deploy only if $\widehat\theta>\tau$,
or restrict to a subgroup chosen after seeing the PPI output), then the relevant
law is conditioned on a data-dependent event. Preserving nominal guarantees in
such pipelines typically requires modeling the enlarged filtration (or using
explicit selective-inference corrections), rather than treating the PPI output
as an updateable predictive state.

\bibliographystyle{apalike}
\bibliography{refs_arxiv}

\end{document}